\documentclass[12pt]{amsart}
\usepackage{color,graphicx,array, amssymb, amscd,slashed}
\usepackage[all]{xy}
\usepackage{tikz-cd}
\usepackage[normalem]{ulem}

\newcommand\soutpars[1]{\let\helpcmd\sout\parhelp#1\par\relax\relax}
\long\def\parhelp#1\par#2\relax{%
  \helpcmd{#1}\ifx\relax#2\else\par\parhelp#2\relax\fi%
}

\newcommand{\toremove}[1]{{}}
\usepackage[backend=bibtex,style=numeric,doi=false,isbn=false,url=false]{biblatex}

\AtEveryBibitem{%
  \clearfield{note}%
}
\renewbibmacro{in:}{}
\DeclareFieldFormat[article,book,inproceedings,thesis,incollection]{title}{#1} 
\bibliography{references.bib}

\newtheorem{Theorem}{Theorem}[section]
\newtheorem{Lemma}[Theorem]{Lemma}
\newtheorem{Proposition}[Theorem]{Proposition}

\theoremstyle{definition}
\newtheorem{Definition}{Definition}
\theoremstyle{remark}
\newtheorem{Example}{Example}
\newtheorem{Remark}[Theorem]{Remark} 

\numberwithin{equation}{section}
\newcommand{\R}{\mathbb R}
\renewcommand{\l}{\lambda}
\newcommand{\C}{\mathbb C}
\newcommand{\Z}{\mathbb Z}
\newcommand{\D}{\mathbb D}

\newcommand{\Q}{\mathbb Q}

\newcommand{\F}{\mathbb F}
\newcommand{\V}{\mathbb V}
\newcommand{\U}{\mathbb U}
\newcommand{\wt}{\widetilde}
\newcommand{\Mp}{\mathcal P}
\newcommand{\f}{\mathfrak f}

\newcommand{\SU}{{\rm SU}_{2}}

\newcommand{\LSU}{\Lambda {\rm SU}_{2,  \sigma}}

\newcommand{\SL}{{\rm SL}_2 \mathbb C}

\newcommand{\LSL}{\Lambda {\rm SL}_2 \mathbb C_{\sigma}}
\newcommand{\LSLN}{\Lambda^{-} {\rm SL}_2 \mathbb C_{\sigma}}
\newcommand{\LSLNN}{\Lambda^{-}_* {\rm SL}_2 \mathbb C_{\sigma}}
\newcommand{\LSLr}{\Lambda_r {\rm SL}_2 \mathbb C_{\sigma}}
\newcommand{\LSLP}{\Lambda^{+} {\rm SL}_2 \mathbb C_{\sigma}}
\newcommand{\LSLPN}{\Lambda^{+}_* {\rm SL}_2 \mathbb C_{\sigma}}

\newcommand{\id}{\operatorname{id}}

\newcommand{\ad}{\operatorname{Ad}}
\newcommand{\di}{\operatorname{diag}}

\usepackage{bbm}
\renewcommand{\Re}{\operatorname {Re}}
\renewcommand{\Im}{\operatorname {Im}}
\newcommand{\qi}{\mathbbm{i}}
\newcommand{\qj}{\mathbbm{j}}
\newcommand{\qk}{\mathbbm{k}}
\usepackage{amsmath}	%
\begin{document}
\title{Discrete constant mean curvature surfaces on general graphs}
 \author[T.~Hoffmann]{Tim Hoffmann}
 \address{Fakult\"at f\"ur Mathematik, 
 TU-M\"unchen, Boltzmann str. 3, D-85747,  Garching, Germany}
 \email{tim.hoffmann@ma.tum.de}

 \author[S.-P.~Kobayashi]{Shimpei Kobayashi}
 \address{Department of Mathematics, Hokkaido University, 
 Sapporo, 060-0810, Japan}
 \email{shimpei@math.sci.hokudai.ac.jp}

 \author[Z.~Ye]{Zi Ye}
 \address{Fakult\"at f\"ur Mathematik, 
 TU-M\"unchen, Boltzmann str. 3, D-85747,  Garching, Germany}
 \email{ye@ma.tum.de}

 \thanks{The authors are partially supported by  Deutsche Forschungsgemeinschaft-Collaborative Research Center, TRR 109, ``Discretization in Geometry and Dynamics'' and
 the second named author is partially supported by JSPS 
 KAKENHI Grant Number JP18K03265.}
 \subjclass[2010]{}
 \keywords{}
 \date{\today}
\pagestyle{plain}
\begin{abstract}
 The contribution of this paper is twofold. First, we generalize 
 the definition of discrete isothermic surfaces. Compared with the 
 previous ones, it covers more discrete surfaces, e.g., 
 the associated families of discrete isothermic minimal and non-zero constant mean 
 curvature (CMC in short) surfaces, whose counterpart in smooth case 
 are isothermic surfaces. Second, we show that the discrete isothermic CMC surfaces can 
 be obtained by the discrete holomorphic data (a solution of the 
 additive rational Toda system) via 
 the discrete generalized Weierstrass type representation.
\end{abstract}
\maketitle

\section*{Introduction}
The study of discrete differential geometry aims at finding good discrete analogues of smooth 
 differential geometric objects. According to previous studies, various discretization of surfaces mostly follow from one of the following two principles: the integrable system and the variational principle. While the former always results in quadrilateral nets which mimic surfaces with a particular parameterization, and the latter often yields discrete nets with more general underlying graph, e.g., discrete 
 minimal surfaces have been defined by the integrable system in \cite{bobenko1996isothermic} 
 and the variational principle in 
\cite{Pinkall@1993}.

 The original works of Bobenko and Pinkall \cite{bobenko1999discretization,bobenko1996isothermic} built a foundation of discrete integrable surfaces, and one particular class of surfaces, called the isothermic surfaces, which includes minimal surfaces, CMC surfaces, quadrics and so on, have drawn much attention.
Specifically, they defined the discrete isothermic parametrized surface from $\mathbb{Z}^2$ by using 
 factorized property of the cross-ratio on a quadrilateral net. In this setting, 
 the discrete isothermic parametrized  minimal and CMC surfaces have been characterized 
 from the particular discrete integrable equations. Moreover, as in the smooth case, the discrete minimal and CMC surfaces have a continuous $S^1$-family, the so-called associated family, which is easily obtained 
 by conservation of the integrable equations under $S^1$-symmetry. While geometric interpretation of the original discrete isothermic parametrized surface was clear, the remaining surfaces in the associated family were not understood until  \cite{Hoffmann2016constraint}, which justified the minimal and CMC associated families by considering a vertex-based normal that satisfies the so-called edge-constraint condition. However, it was still unknown how to account for the isothermicity of all the surfaces in the associated family. We will briefly recall  isothermic parametrized constant mean curvature surfaces on quadrilateral graphs in Section \ref{sc:isopara}.

On the one hand, Lam and Pinkall \cite{Lam2016isothermic} attempted to define discrete isothermic surfaces without referring to a particular parameterization; 
 they defined a discrete isothermic triangulated 
surface which was related to an infinitesimal deformation that preserves the mean curvature.  Moreover, in \cite{Lam2016minimal}  
 Lam found a continuous deformation between the minimal surfaces from the integrable system \cite{bobenko1996isothermic} 
 and the variational principle \cite{Pinkall@1993}. However, this definition did not cover the whole associated families of a minimal surface or  a CMC surface either as 
discrete isothermic triangulated surfaces. In 
 \cite{Hoffmann2018face}, the first and the third named authors reformulated the these minimal surfaces with a discrete Dirac operator and gave a hint of more general isothermic surfaces. Finally this paper will finally fill the gap and obtain the isothermicity that 
 covers all discrete surfaces of the associated families of a minimal surface and a CMC surface, see in Section \ref{sc:Mainresults} and Section \ref{sc:isopara}
 in details.

It is well-known that there is a correspondence, called the Weierstrass-Enneper representation, between holomorphic functions and smooth minimal surfaces. In discrete case, it has been known that holomorphic functions on quadrilateral graphs are understood as the cross-ratio system. In \cite{bobenko1996isothermic}, Bobenko and Pinkall showed that discrete isothermic parametrized minimal surfaces could be induced from the cross-ratio system, and it has been called the discrete Weierstrass representation. For the case of CMC surfaces, the relation to holomorphic functions is much less straightforward, however in \cite{Dorfmeister_1998}, Dorfmeister, Pedit and Wu  showed that one could construct smooth CMC surfaces from the holomorphic data with loop group decompositions, and it has been called the generalized Weierstrass type representation 
 or the DPW method. In \cite{Hoffmann1999cmc} the first named author obtained a discrete 
 analogues of the DPW method for the discrete isothermic parametrized CMC surfaces on 
 $\mathbb Z^2$, i.e., 
 the CMC surfaces can be obtained from the cross-ratio system and loop group decompositions, see Section \ref{sc:DPW}.

On more general graphs, the holomorphic functions can be modeled by the discrete holomorphic quadratic differential \cite{Lam2016}, which is closely related to the additive rational Toda system \cite{Bobenko2002quad}.  In \cite{Lam2016minimal}, the first and the third named authors showed that this holomorphic quadratic differential indeed induced the discrete minimal surfaces. 
Moreover in \cite{Bobenko2002quad}, Bobenko and Suris remarkably found  that the cross-ratio system on 
a quadrilateral net could be transformed into the additive rational Toda system on a half graph given by the 
quadrilateral graph. Since the cross-ratio system corresponds to discrete CMC surfaces on quadrilateral 
graph through the DPW method \cite{Hoffmann1999cmc}, this hints how to generalize the discrete DPW method to discrete CMC surfaces 
on general graphs. Therefore in this paper, we will show how to generate discrete isothermic CMC surfaces on general graphs from the additive rational Toda system through the straightforward generalization of the discrete DPW method. Furthermore, we show that the discrete Weierstrass representation and the DPW method can be unified by introducing a mean curvature parameter $H$ in the holomorphic data. The mean curvature parameter $H$ varies continuously from $1$ to $0$, giving extended frames for different types of surfaces. In particular, the extended frames with non-zero $H$ induce discrete CMC surfaces via the Sym-Bobenko formula, while the frames with $H=0$ can be solved explicitly and yield minimal surfaces, i.e, we obtain the discrete Weierstrass representation. Schematically, 
we summarize the relations between various previous results and our results in Figure  \ref{fig:scheme}. We will explain the DPW method for discrete isothermic 
 CMC surfaces on general graphs in 
 Section \ref{sc:Iso} in details.
\begin{figure}[ht]
\centering
\def\svgwidth{0.53\textwidth}
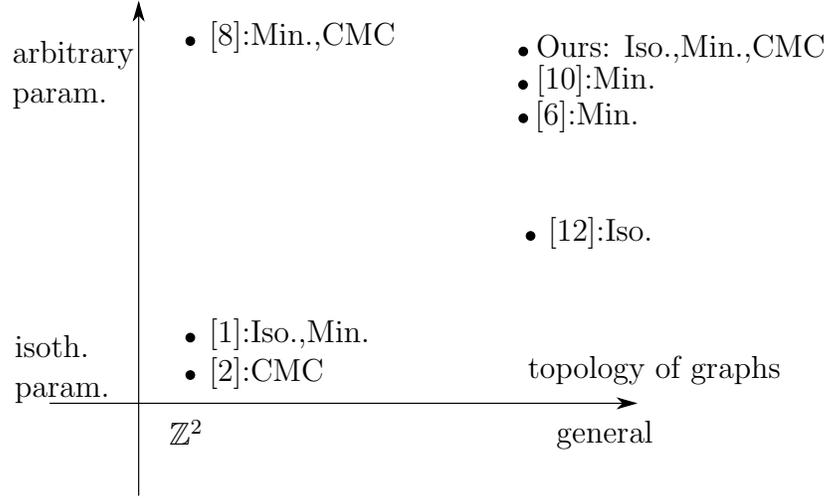
\caption{Existing works on discrete isothermic surfaces (Iso. in short), minimal surface (Min. in short) and CMC surfaces.   }
\label{fig:scheme}
\end{figure}

\section*{Convention}
Notations:

\begin{align*}
\mathcal{G} : \quad &\mbox{A cellular decomposition of an oriented surface} \\
\mathcal{G}^* :\quad &\mbox{The dual graph of $\mathcal{G}$}\\
\mathcal D: \quad &\mbox{The even quadrilateral graph given by $\mathcal G$ and $\mathcal G^*$}\\
\mathcal V(\mathcal G), \mathcal V(\mathcal D): \quad &\mbox{A set of vertices of 
 $\mathcal G$, or $\mathcal D$}\\
\mathcal E(\mathcal G), \mathcal E(\mathcal D): \quad &\mbox{A set of edges of 
 $\mathcal G$, or $\mathcal D$}\\
\mathcal F (\mathcal G), \mathcal F (\mathcal D): \quad &\mbox{A set of faces of 
 $\mathcal G$, or $\mathcal D$}\\
\Phi_-(v_i) :\quad &\mbox{The wave function of holomorphic data at $v_i \in \mathcal V (\mathcal D)$}\\
\Phi(v_i) :\quad &\mbox{The extended frame at $v_i \in \mathcal V (\mathcal D)$}\\
L_-(v_j,v_i) :\quad &\mbox{Transition matrix of holomorphic wave function }\\
&\mbox{from $v_i \in \mathcal V (\mathcal D)$ to $v_j\in \mathcal V (\mathcal D)$ }\\
U(v_j,v_i) :\quad &\mbox{Transition matrix of extended frame from $v_i\in \mathcal V (\mathcal D)$ to $v_j\in \mathcal V (\mathcal D)$} \\
\mathcal{P}_-(v_i) , \mathcal{P}_-^*(v_i) :\quad &\mbox{The wave function of gauged holomorphic data at $v_i\in \mathcal V (\mathcal G)$ }\\
&\mbox{or at $v_i\in \mathcal V (\mathcal G^*)$}\\
\mathcal{P}(v_i), \mathcal{P}^*(v_i) :\quad & \mbox{The gauged extended frame at $v_i\in \mathcal V (\mathcal G)$ or at $v_i\in \mathcal V (\mathcal G^*)$ }\\
\mathcal{L}_-(v_j,v_i) :\quad &\mbox{Transition matrix of gauged holomorphic wave function} \\
& \mbox{from $v_i \in \mathcal V (\mathcal G)$ to $v_j \in \mathcal V (\mathcal G)$} \\
\mathcal{U}(v_j,v_i) :\quad &\mbox{Transition matrix of gauged extended frame from $v_i$ to $v_j$}
\end{align*}

Throughout the paper we always consider the discrete surfaces with the following two types of underlying graphs. First, the even quadrilateral graphs, denoted by $\mathcal{D}$, meaning that any closing loop consists of even edges. Second, the general graphs, denoted by $\mathcal{G}$.

From \cite{Bobenko2002quad} we know that any even quadrilateral graph can be bipartite decomposed into a general graph $\mathcal{G}$ and its dual graph $\mathcal{G}^*$, $\mathcal{V}(\mathcal{D}) = \mathcal{V}(\mathcal{G}) \cup \mathcal{V}(\mathcal{G}^*)$ (see Figure.~\ref{fig:splitting}). Conversely, a general graph $\mathcal{G}$ and its dual $\mathcal{G}^*$ can induce a even quadrilateral graph.

We always identify vectors in $\mathbb{R}^3$ with pure imaginary quaternions,  $\mathbb{R}^3 \cong \operatorname{Im} \mathbb{H}$ by $(x,y,z) \mapsto x\qi+y\qj+z\qk$. Hence $\mathbb{R}^3$-valued vectors are endowed with a multiplicative structure from the quaternion. In some circumstances, we identify quaternions with $2\times 2$ complex matrices:
\[\sigma_1 = \begin{pmatrix} 0 & 1\\ 1 & 0 \end{pmatrix} = \sqrt{-1} \qi,\quad \sigma_2 = \begin{pmatrix} 0 & -\sqrt{-1}\\ \sqrt{-1} & 0 \end{pmatrix} = \sqrt{-1} \qj,\quad \sigma_3 = \begin{pmatrix} 1 & 0\\ 0 & -1 \end{pmatrix} = \sqrt{-1} \qk.\]

\section{Main results}\label{sc:Mainresults}
 In this section we demonstrate necessary definitions and main 
 results of this paper.

\subsection{Discrete isothermic surfaces and constant mean curvature surfaces on general graphs}\label{subsc:discreteisothermic}
It has been known that in smooth case a surface $f:M\rightarrow \R^3$ is called isothermic if  it has a conformal curvature line parametrization. This parametrization is called the 
\textit{isothermic parametrization}. In this paper, we adopt an equivalent definition which does not depend on a particular parameterization.
 
Recall that in \cite{Kamberov_1998}, a surface $f:M\rightarrow \mathbb{R}^3$ is called \textit{isothermic} if there exists a $\mathbb{R}^3$-valued one-form such that 
\[df \wedge \omega = 0,\]
where $\wedge$ is understood as the wedge product for quaternion-valued one-forms.
 From now on we only consider the local theory, i.e., $M$ is simply 
 connected,  
 hence one can find a surface $f^*:M\rightarrow \mathbb{R}^3$ such that 
\begin{equation}
\label{eqn:def_isothermic}
df \wedge d f^* = 0
\end{equation}
holds, 
which can be expressed with the Dirac operator $D_f$
\[D_f f^*  = 0 \quad \mbox{for} 
\quad D_f := \frac{d f \wedge d}{|df|^2}.\]
We call $f^*$ the \textit{Christoffel dual} of $f$. 
 
Following the defining equation \eqref{eqn:def_isothermic} of a smooth isothermic surface, we define a discrete isothermic surface via a discrete Christoffel dual. Unlike the Christoffel dual for quadrilateral nets \cite{bobenko1996isothermic}, whose graph has the same topology with the primal net, our Christoffel dual is defined on the dual graph. In order to cover the known examples, one has to consider the quaternion with \textit{real part} which does not show up in the smooth case.

We use of the standard notation $f_i = f(v_i)$, $f_i^* = f^*(v_i)$ for maps of vertices,
 $E_{ij} = E(e_{ij})$, $E_{ij}^* = E^*(e_{ij})$ for maps of edges, and the notion from the discrete exterior calculus, i.e., the discrete one-form $df_{ij} := df(e_{ij}) = f(v_j) - f(v_i)$ with $v_i, v_j \in \mathcal V(\mathcal G)$ and the dual one-form $df^*_{ij} :=df^*(e_{ij}^*) = f^*(v_j) - f^*(v_i)$with $v_i, v_j \in \mathcal V(\mathcal G^*)$.
\begin{Definition}[Isothermic surfaces on general graphs]\label{def:genIso}
Let $f: \mathcal V(\mathcal G)\rightarrow \mathbb{R}^3$ be a discrete surface. We call $f$ an \textit{isothermic} surface if there exists a surface 
$f^*:\mathcal V(\mathcal G^*) \rightarrow \mathbb{R}^3$ and real-valued functions 
$R:\mathcal E(\mathcal G)\rightarrow \mathbb{R}$ and $R^*:\mathcal E(\mathcal G^*) \rightarrow \mathbb{R}$ such that 
\[
 E_{ij} = R_{ij} + df_{ij}\quad \mbox{and} \quad 
 E^*_{ij} = R_{ij}^* + df^*_{ij}
\] 
 satisfy the following equations:
 \begin{equation}
 \label{eq:dis_iso} \sum_j E_{ij} \cdot E^*_{ij}= 0 \quad \mbox{for all 
 $i$},
\end{equation}
where the sum runs over every face of $f$ (see Figure.~\ref{fig:dis_iso}).
 Moreover, the surface $f^*$ will be called the \textit{Christoffel dual}
 of $f$, and the functions $E_{ij}$ and $E_{ij}^*$ associated with
 the edge $e_{ij} \in \mathcal E(\mathcal G)$ and 
 the edge $e_{ij}^* \in \mathcal E(\mathcal G^*)$ will be called the 
\textit{hyperedges}.
\end{Definition}

\begin{figure}[ht]
\centering
\def\svgwidth{0.4\textwidth}
\begingroup%
  \makeatletter%
  \providecommand\color[2][]{%
    \errmessage{(Inkscape) Color is used for the text in Inkscape, but the package 'color.sty' is not loaded}%
    \renewcommand\color[2][]{}%
  }%
  \providecommand\transparent[1]{%
    \errmessage{(Inkscape) Transparency is used (non-zero) for the text in Inkscape, but the package 'transparent.sty' is not loaded}%
    \renewcommand\transparent[1]{}%
  }%
  \providecommand\rotatebox[2]{#2}%
  \newcommand*\fsize{\dimexpr\f@size pt\relax}%
  \newcommand*\lineheight[1]{\fontsize{\fsize}{#1\fsize}\selectfont}%
  \ifx\svgwidth\undefined%
    \setlength{\unitlength}{154.12496964bp}%
    \ifx\svgscale\undefined%
      \relax%
    \else%
      \setlength{\unitlength}{\unitlength * \real{\svgscale}}%
    \fi%
  \else%
    \setlength{\unitlength}{\svgwidth}%
  \fi%
  \global\let\svgwidth\undefined%
  \global\let\svgscale\undefined%
  \makeatother%
  \begin{picture}(1,0.87507418)%
    \lineheight{1}%
    \setlength\tabcolsep{0pt}%
    \put(0,0){\includegraphics[width=\unitlength,page=1]{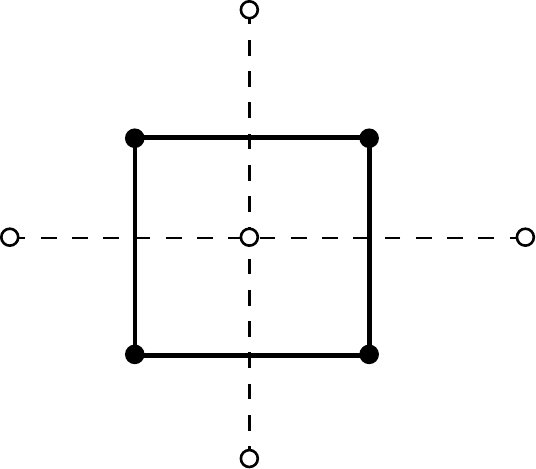}}%
    \put(0.48044889,0.44999606){\color[rgb]{0,0,0}\makebox(0,0)[lt]{\lineheight{1.25}\smash{\begin{tabular}[t]{l}$i$\end{tabular}}}}%
    \put(0.92083817,0.46946081){\color[rgb]{0,0,0}\makebox(0,0)[lt]{\lineheight{1.25}\smash{\begin{tabular}[t]{l}$j$\end{tabular}}}}%
    \put(0,0){\includegraphics[width=\unitlength,page=2]{isoth.pdf}}%
    \put(0.56560711,0.52298864){\color[rgb]{0,0,0}\makebox(0,0)[lt]{\lineheight{1.25}\smash{\begin{tabular}[t]{l}$e_{ij}$\end{tabular}}}}%
    \put(0.81378251,0.35023922){\color[rgb]{0,0,0}\makebox(0,0)[lt]{\lineheight{1.25}\smash{\begin{tabular}[t]{l}$e^*_{ij}$\end{tabular}}}}%
  \end{picture}%
\endgroup%

\caption{Edges and the dual edges}\label{fig:dis_iso}
\end{figure}
\begin{Remark}
Our definition generalizes the notion of isothermic surfaces defined 
by Lam and Pinkall \cite{Lam2016isothermic}.  Reformulating the quaternion in \eqref{eq:dis_iso} in terms of scalar product and cross product in $\mathbb{R}^3$,
 we obtain:
\begin{align}
\label{eq:dis_iso1}
\sum_j R_{ij}R^*_{ij}-\langle df_{ij} , df^*_{ij}\rangle &= 0 \quad 
\mbox{for all $i$}, \\
\label{eq:dis_iso2}
\sum_j R_{ij} df^*_{ij}+ R^*_{ij} df_{ij} + df_{ij} \times df^*_{ij} &= 0 \quad 
\mbox{for all $i$}.
\end{align}
Recall that the discrete isothermic net in \cite{Lam2016isothermic} can be formulated by
\begin{align}
\label{eq:dis_iso_lam1}
\sum_j \langle df_{ij} , df^*_{ij}\rangle &= 0\quad \mbox{for all $i$}, \\
\label{eq:dis_iso_lam2}
df_{ij} \times df^*_{ij} &= 0 \quad \mbox{for all edges $e_{ij}$}.
\end{align}
Clearly, \eqref{eq:dis_iso1} is a generalization of \eqref{eq:dis_iso_lam1} and \eqref{eq:dis_iso2} is a generalization of \eqref{eq:dis_iso_lam2}. They will be equivalent if all the real parts vanish and all the dual edges are parallel. We will show that our definition covers the associated family of discrete minimal surfaces and CMC surfaces, whereas the definition of \cite{Lam2016isothermic} only includes the isothermic parametrized ones.
\end{Remark}  
In smooth case, it is well known that the M\"obius transformation of isothermic surface is still isothermic. In discrete case, the notions of isothermicity \cite{bobenko1996isothermic,Lam2016isothermic} prove to be M\"obius invariant. In our setting, when the dual edges are all associated with zero real part, i.e., 
 $R_{ij}^* = 0$ (which is satisfied by most existing cases, e.g., the associated families of isothermic minimal and CMC surfaces induced by holomorphic functions with cross-ratio being $-1$), the M\"obius invariance is understood in the following sense.

\begin{Proposition}[M\"obius invariance of isothermic surfaces]
Let $f$ be a discrete isothermic surface which is dual to $f^*$ such that $R_{ij}^* =0$ for all dual edges. Then the M\"obius transformation of $f^*$ is dual to the isothermic surface $\tilde{f}$ obtained by $\tilde{E}_{ij}:= \overline{f^*_i}\cdot E_{ij}\cdot  f_j^*$.
\end{Proposition}
\begin{proof}
We first show the closing condition holds for $\tilde{f}$.  In fact,
\[
\sum_j \overline{f^*_i}\cdot E_{ij} \cdot f_j^* = \sum_j \overline{f^*_i}\cdot E_{ij} \cdot f^*_i = \lvert f_i^* \rvert^2 \sum_j R_{ij} \in \mathbb{R},
\]
where we have used \eqref{eq:dis_iso} in the first equality and 
$\sum_j d f_{ij} =0$ (see Figure \ref{fig:dis_iso}) in the second equality.
The last term implies that the edges of $\tilde{f}$ sum up to $0$ for every faces.
The M\"obius transformation of $f^*$ can be written as the quaternion $f^*\mapsto (f^*)^{-1}$. It follows that
\[d((f^*)^{-1})_{ij} = (f^*_j)^{-1} - (f^*_i)^{-1} = -(f_j^*)^{-1} \cdot df^*_{ij} \cdot (f_i^*)^{-1}.\]
Then we check the isothermic condition,
\[\sum_j \tilde{E}_{ij} \cdot  d((f^*)^{-1})_{ij} =  \overline{f^*_i}\cdot E_{ij} \cdot f_j^* \cdot ( (f^*_j)^{-1} \cdot df_{ij}^* \cdot f^*_i) = \overline{f_i^*} \cdot \left( \sum_j E_{ij} \cdot df_{ij}^*\right) \cdot f_i^* =0. \] 
This completes the proof.
\end{proof}
 
\begin{Example}[Isothermic parametrized surfaces \cite{bobenko1996isothermic}]\label{ex:iso}
 A discrete surface $\mathfrak f:\mathbb{Z}^2\rightarrow \mathbb{R}^3$ is called \textit{isothermic parametrized} if  it has factorized cross-ratios. 
  It is known any discrete isothermic parametrized surface admits a Christoffel dual $\mathfrak f^*:\mathbb{Z}^2 \rightarrow \mathbb{R}^3$ with the same underlying graph, see Appendix \ref{subsc:iso} for more details.

A splitting of $\mathbb{Z}^2$ into two graphs can be obtained by decomposing the vertices into black vertices $\mathcal V:=\{(x,y)\;|\;x+y  = 0 \mod 2\}$  
and white vertices $\mathcal V^* := \{(x,y)\;|\; x+y = 1\mod 2\}$ and by taking the diagonals of $\mathbb{Z}^2$ as the new edges, see Figure \ref{fig:splitting}. Then, one obtains two discrete nets, which are topologically dual to each other. By restricting $f:= \mathfrak {f}|_{\mathcal V(\mathcal G)}$ and $f^*:=\mathfrak {f}^*|_{\mathcal V(\mathcal G^*)}$, it is easy to verify that $f$ and $f^*$ satisfy the definition \eqref{eq:dis_iso} with ${R}_{ij} =0$ and ${R}_{ij}^* =0$ for 
 all edges $e_{ij}$ and $e_{ij}^*$.

\begin{figure}[ht]
\centering
\includegraphics[scale=0.4]{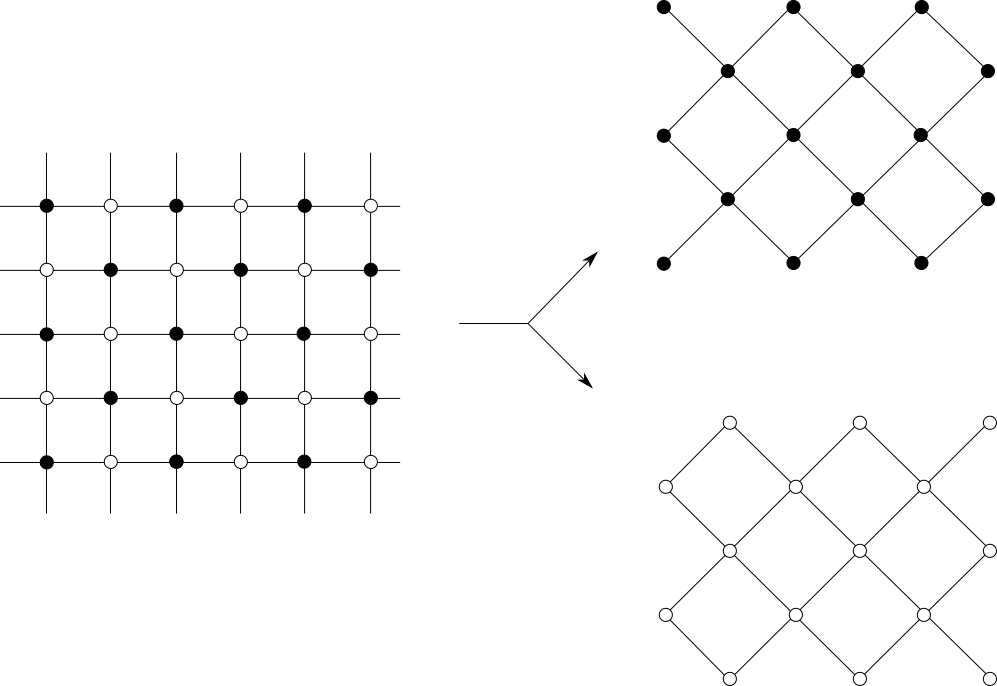}
\caption{Decomposition of $\mathbb{Z}^2$ into graphs with black and white vertices.}\label{fig:splitting}
\end{figure}

\end{Example}

\begin{Remark} 
 The same procedure as above can be taken to generate isothermic parametrized 
 constant mean curvature surfaces (both minimal and CMC) in our setting from the traditional 
 ones over $\mathbb{Z}^2$. 
However, general members in 
 the associated family of an isothermic parametrized 
 constant mean curvature surface
 do not have the factorized cross-ratio property, 
 and thus they are not 
 an isothermic parametrized surface. Moreover, $R_{ij}$ and $R_{ij}^*$
 are not zero, thus they are not even an isothermic surface in the sense of 
 Lam and Pinkall \cite{Lam2016isothermic}.
\end{Remark}

\begin{Example}[Face-edge-constraint minimal surfaces]
\label{ex:minimal}
Lam \cite{Lam2016minimal} showed that two types of discrete minimal surfaces, namely the A-minimal surface from integrable system and the C-minimal surface obtained from variational principle, can be related with an associated family. 
In \cite{Hoffmann2018face} showed that the family of minimal surfaces can be interpreted by the face-edge-constraint minimal surfaces. Specifically, a discrete surface $f:\mathcal V(\mathcal G)\rightarrow \mathbb{R}^3$ with normal defined on faces $n:\mathcal F(\mathcal G)\rightarrow S^2$ is called \textit{face-edge-constraint} if 
\begin{equation}\label{eq:edge_constraint}
(n_i + n_j) \perp df_{ij}
\end{equation}
holds for every edges $e_{ij} \in \mathcal E(\mathcal G)$. For such nets one can define the 
\textit{integrated mean curvature} for edges by 
\[\mathbf{H}_{ij} = \frac{1}{2} \lvert df_{ij}\rvert \tan \frac{\theta_{ij}}{2},\]
where $\theta_{ij}$ is the bending angle between the planes $P_i:=\mathrm{span}\{n_i,df_{ij}\}$ and $P_j:=\mathrm{span}\{n_j,df_{ij}\}$. The face-edge-constraint minimal surface in this setting is naturally defined by the surface with vanishing integrated mean curvature for every faces, i.e., $\mathbf{H}_i := \sum_{j} \mathbf{H}_{ij} = 0$.

By a simple calculation (Proposition 3.8 in \cite{Hoffmann2018face}) we have
\begin{equation}
\label{eq:edge_constraint_var}
E_{ij}^{-1}\cdot n_i \cdot E_{ij} = -n_j.
\end{equation}
Let the real part $R_{ij} := 2\mathbf{H}_{ij}$ and $R^*_{ij}=0$, i.e., 
$E_{ij} = 2 \mathbf H_{ij} + df_{ij}$ and 
$E_{ij}^* = df^*_{ij}$. Moreover, define $f^*=n$.
It is easy to show that the face-edge-constraint minimal surface satisfies the isothermic condition with the normal being the dual surface:
\[\sum_j E_{ij} \cdot E^*_{ij} = \sum_j E_{ij} \cdot (n_j -n_i) = \sum_j ( -n_i \cdot E_{ij} - E_{ij} \cdot n_i)=0 \quad \mbox{for all $i$}, \] 
where for the second equality we use \eqref{eq:edge_constraint_var}
and $\sum_j E_{ij}=0$ in the last equality which follows from the minimality 
 condition $\mathbf H_i = \sum_j \mathbf H_{ij} = \sum_j R_{ij}/2=0$.
\end{Example}

Given any discrete net $f:\mathcal V(\mathcal G) \rightarrow \mathbb{R}^3$ and its dual $f^*:\mathcal V(\mathcal G^*)\rightarrow \mathbb{R}^3$, one can associate each edge, denoted by $(f_0, f_{2})$, and its dual edge, denoted by $(f^*_1,f^*_3)$, with an elementary quadrilateral with vertices $f_0$, $f^*_1$, $f_2$ and $f^*_3$.

 Together with the definition of isothermic surfaces in Definition 
 \ref{def:genIso},  the following definition for isothermic 
 constant mean curvature surfaces is the 
 main results in this paper.
\begin{Definition}[Geometric definition of constant mean curvature
 surfaces on general graphs]
\label{def:constantmean}
 Let $f$ and $f^*$ be a discrete isothermic surface and its Christoffel dual 
 surface in Definition \ref{def:genIso}.
 Moreover,  for every elementary quadrilateral  $(v_0, v_1, v_2, v_3)$ on 
 $\mathcal D$ (constructed by $\mathcal G$ and $\mathcal G^*$)  with $v_0, v_2 \in \mathcal V(\mathcal G)$ and $v_1, v_3 
 \in \mathcal V(\mathcal G^*)$, denote  $E_{02}(= R_{02} + d f_{02})$ to be 
 the hyperedge associated with the edge 
 $e_{02} \in \mathcal E (\mathcal G)$.
\begin{enumerate}
\item ($f$, $f^*$) is called \textit{a  pair of non-zero constant mean curvature} surfaces 
(CMC surfaces in short) if the following two equations hold$:$
\begin{equation}
\label{eq:diag_trans}
\left\{
\begin{array}{l}
\displaystyle f^*_{3} - f_2 = E^{-1}_{02} \cdot (f_0-f^*_1) \cdot E_{02}, \\[0.2cm]
\displaystyle f^*_{3} - f_0 = E^{-1}_{02} \cdot (f_2-f^*_1) \cdot E_{02}. 
\end{array}
\right.
\end{equation}
\item $f$ is called a \textit{minimal} surface if 
 $f^*$ takes values in the unit two sphere $S^2$ and
 the following two equations hold$:$
\begin{equation}
\label{eq:diag_trans_min}
\left\{
\begin{array}{l}
 f^*_{3} = -E^{-1}_{02} \cdot f^*_1 \cdot E_{02},  \\[0.2cm]
 \sum_{j} R_{ij}=0\quad \mbox{hold for all $i \in \mathcal V(\mathcal G)$,}
\end{array}
\right.
\end{equation}
 where the sum takes such that $j$ are vertices adjacent to the vertex $i$.
\end{enumerate}
\begin{figure}[ht]
\centering
\def\svgwidth{0.4\textwidth}
\begingroup%
  \makeatletter%
  \providecommand\color[2][]{%
    \errmessage{(Inkscape) Color is used for the text in Inkscape, but the package 'color.sty' is not loaded}%
    \renewcommand\color[2][]{}%
  }%
  \providecommand\transparent[1]{%
    \errmessage{(Inkscape) Transparency is used (non-zero) for the text in Inkscape, but the package 'transparent.sty' is not loaded}%
    \renewcommand\transparent[1]{}%
  }%
  \providecommand\rotatebox[2]{#2}%
  \newcommand*\fsize{\dimexpr\f@size pt\relax}%
  \newcommand*\lineheight[1]{\fontsize{\fsize}{#1\fsize}\selectfont}%
  \ifx\svgwidth\undefined%
    \setlength{\unitlength}{292.18654889bp}%
    \ifx\svgscale\undefined%
      \relax%
    \else%
      \setlength{\unitlength}{\unitlength * \real{\svgscale}}%
    \fi%
  \else%
    \setlength{\unitlength}{\svgwidth}%
  \fi%
  \global\let\svgwidth\undefined%
  \global\let\svgscale\undefined%
  \makeatother%
  \begin{picture}(1,0.58338099)%
    \lineheight{1}%
    \setlength\tabcolsep{0pt}%
    \put(0,0){\includegraphics[width=\unitlength,page=1]{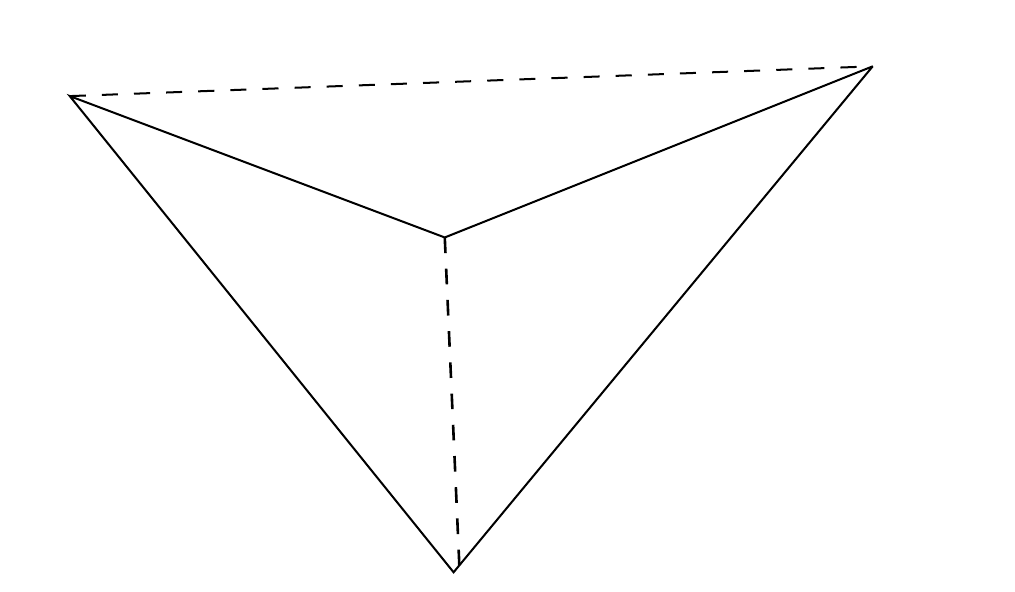}}%
    \put(-0.00284092,0.52914271){\color[rgb]{0,0,0}\makebox(0,0)[lt]{\lineheight{1.25}\smash{\begin{tabular}[t]{l}$f_0$\end{tabular}}}}%
    \put(0.86475558,0.55737823){\color[rgb]{0,0,0}\makebox(0,0)[lt]{\lineheight{1.25}\smash{\begin{tabular}[t]{l}$f_{2}$\end{tabular}}}}%
    \put(0.47459381,0.00807155){\color[rgb]{0,0,0}\makebox(0,0)[lt]{\lineheight{1.25}\smash{\begin{tabular}[t]{l}$f^*_1$\end{tabular}}}}%
    \put(0.46689328,0.29812599){\color[rgb]{0,0,0}\makebox(0,0)[lt]{\lineheight{1.25}\smash{\begin{tabular}[t]{l}$f^*_3$\end{tabular}}}}%
    \put(0.49549244,-2.07313001){\color[rgb]{0,0,0}\makebox(0,0)[lt]{\begin{minipage}{0.29104478\unitlength}\raggedright \end{minipage}}}%
    \put(0,0){\includegraphics[width=\unitlength,page=2]{parallelogram.pdf}}%
  \end{picture}%
\endgroup%

\caption{An elementary parallelogram.}\label{fig:parallelogram}
\end{figure}
\end{Definition}
\begin{Remark}
\mbox{}
\begin{enumerate}
\item If $(f, f^*)$ is a pair of discrete isothermic CMC
surfaces with hyperedges $E_{ij} = R_{ij} + df_{ij}$ and $E^*_{ij} = R^*_{ij} +
df^*_{ij}$,  then, $(f^*, f)$ is also a pair of discrete isothermic CMC surface with 
 hyperedges $E^*_{ij} =
R_{ij} + df^*_{ij}$ and $E_{ij} = R^*_{ij} + df_{ij}$.
\item
Definition \ref{def:constantmean} (1) implies that the quadrilateral has the opposite edges with equal length, i.e., $\lvert f^*_1 -f_0\rvert = \lvert f_{2} -f_3^*\rvert$ and $\lvert f_3^* - f_0 \rvert = \lvert f_{2} -f_1^*\rvert$.
 Moreover, $f_0$, $f^*_1$, $f_{2}$ and $f^*_3$ form a \textit{equally-folded skew parallelogram}, see Lemma 40 in \cite{Hoffmann2016constraint}, 
 see Figure \ref{fig:parallelogram}.

\item In Theorem \ref{thm:extendedandCMC} and Theorem 
 \ref{thm:extendedandminimal}, we will show that 
 each member  in the associated family of an isothermic 
 constant mean curvature surface
 is an isothermic constant mean curvature surface in the sense of 
 Definition \ref{def:constantmean}.

\item In Theorem \ref{thm:extendedandminimal}, we will show that a discrete 
 CMC surface naturally converges to a discrete minimal surface 
 in the sense of Definition \ref{def:constantmean}
 when the mean curvature parameter $H$ goes to $0$.
\end{enumerate}
\end{Remark}

\subsection{Discrete holomorphic function: From cross-ratio systems to additive rational Toda systems}\label{sbsc:Toda} It is known \cite{Dorfmeister_1998}
 that smooth constant mean curvature surfaces can be constructed by holomorphic data (the Weierstrass type representation formula), 
 which is the classical Weierstrass representation for minimal surfaces and 
 the so-called DPW representation for CMC surfaces, respectively.

 On quad graphs, a discrete analogue of this representation has been known in \cite{bobenko1996isothermic, Hoffmann1999cmc}.  We generalize this representation  to discrete isothermic constant  mean curvature surfaces on general graphs.

We first recall notion of discrete holomorphic functions on quad graphs.
  \begin{Definition}\label{def:cross_ratio}
\mbox{}
\begin{enumerate}
\item A function $\alpha:\mathcal E(\mathcal{D}) \rightarrow \C$ 
 is called a \textit{labelling} if 
 $\alpha (\mathfrak e) = \alpha (\mathfrak e^*)$ for any edge 
 $\mathfrak e \in \mathcal E(\mathcal{D})$ and 
 the values of two opposite edges on any quadrilateral are equal. 
\item 
 Let $\alpha : \mathcal E (\mathcal D) \to \C$
 be a labelling. Then a following system is called the {\it cross-ratio system}: 
\begin{equation}\label{eq:def-cross}
 q(z_0,z_1,z_2,z_3) :=  \frac{(z_0 -z_1)(z_2-z_3)}{(z_1 -z_2)(z_3 - z_0)}
 =\frac{\alpha_1}{\alpha_2}
\end{equation}
 holds for any elementary quadrilateral $(v_0, v_1, v_2, v_3)$.
 Here we write $z_i = z(v_i) \;(i=0, 1, 2, 3)$, 
$\alpha_1 = \alpha(\mathfrak e_{01}) = \alpha( \mathfrak e_{32})$ and 
 $\alpha_2 = \alpha(\mathfrak e_{03}) = \alpha( \mathfrak 
 e_{12})$ for $\mathfrak e_{ij} \in \mathcal E(\mathcal D)$.
\end{enumerate}
\end{Definition}
 
It is known \cite{Hoffmann1999cmc} that a solution  
 $z: \mathcal V (\mathcal D) \to \C$
 of the cross-ratio system 
 \eqref{eq:def-cross} with $\alpha_1/\alpha_2 =-1$
is called the \textit{discrete holomorphic function}. By abuse of notation we also 
call any solution of the 
 cross-ratio system \eqref{eq:def-cross} 
 the \textit{discrete holomorphic function}.

 In \cite{bobenko1996isothermic} (for the case of minimal surfaces) 
 and \cite{Hoffmann1999cmc} (for  the case of CMC surfaces), 
 it has been proved that every solution of a cross-ratio 
 system on $\mathcal D = \mathbb Z^2$
 gives an
 isothermic parametrized constant mean curvature surface.
 Conversely,  the extended frame of an isothermic parametrized constant mean curvature surface on 
 $\mathcal D$
 gives a solution of a cross-ratio system.

 Here the extended frame  is 
 a map from $\mathcal V (\mathcal D)$ into the 
 loop group $\LSU$ such that it gives naturally 
 a constant mean 
 curvature surface through the Sym-Bobenko formula 
 (for the case of CMC surfaces) or a direct summation formula
 (for the case of minimal surfaces), see Section \ref{sbsc:extended}.
 
 Moreover there were assumed that $\mathcal D$ is the square lattice $\Z^2$ and the labelling $\alpha$ takes values in $\R^{\times}$ and satifies
 \[
  \frac{\alpha_1}{\alpha_2}<0.
 \]
 In Section \ref{sc:isopara}, we will generalize them to  arbitrary quad-graphs $\mathcal D$
 and arbitrary labellings  $\alpha$ which take values in $\C^{\times}$, and introduce the mean curvature parameter $H$. 

\begin{figure}[ht]
\def\svgwidth{0.8\textwidth}
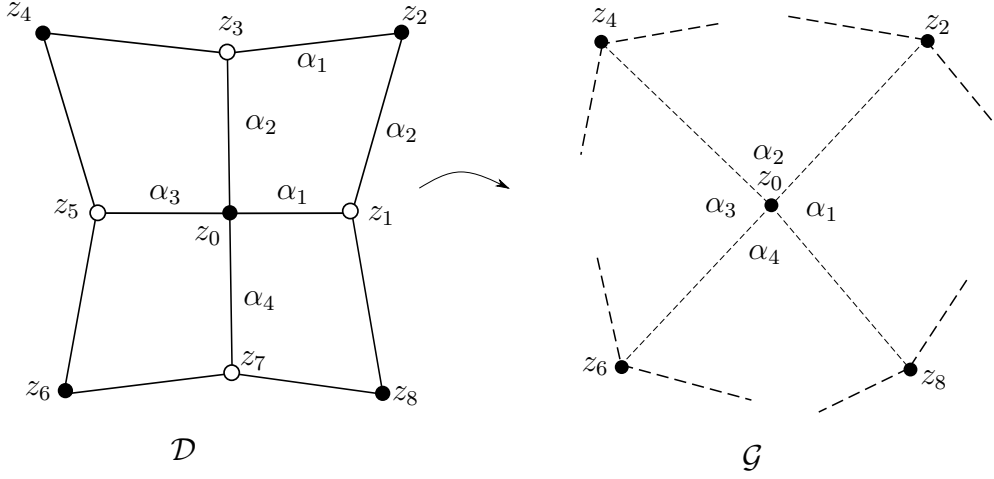
\caption{Bipartite decomposition}\label{fig:bipartite}
\end{figure}

 It is known \cite{Bobenko2002quad} that 
  the cross-ratio system \eqref{eq:def-cross} can be transformed as  the 
 \textit{three-leg form}:
\begin{equation}\label{eq:threeleg}
 \frac{\alpha_1}{z_0 -z_1} -  \frac{\alpha_2}{z_0-z_3}  
 =  \frac{\alpha_1 - \alpha_2}{z_0 -z_2}.
\end{equation}
 Then adding all three-leg forms as in \eqref{eq:threeleg} 
 around the vertex $v_0$, see Figure 
\ref{fig:bipartite}, 
 we obtain the equation which depends only on the filed $z$ 
in the \textit{black vertices} $v_{2k}$: 
\begin{equation}\label{eq:additiverationaltoda}
 \sum_{k=1}^n \frac{\alpha_k - \alpha_{k+1}}{z_0 - z_{2k}} = 0.
\end{equation}
 Thus on general graphs, notion of holomorphicity is defined as follows.
\begin{Definition}[\cite{Bobenko2002quad}]
Let $\mathcal{G}$ be any oriented cell decomposition of a surface.
\begin{enumerate}
\item  The \textit{corner} $\mathcal{C}(\mathcal{G})$ is a set of pairs $(i,m)$ of vertices $i\in \mathcal V(\mathcal G)$ and faces $m\in \mathcal F (\mathcal G)$ such that $i$ is incident to $m$. 
\item The \textit{additive rational Toda system} on $\mathcal G$ is a map $z:\mathcal{V} (\mathcal G)\rightarrow \C$ together with a corner map $\alpha:\mathcal{C}(\mathcal G)\rightarrow \C$ such that 
\begin{equation}
\label{eqn:toda1}
\alpha_{i,m} = \alpha_{j,n}\quad \text{for all edges $e_{ij} \in \mathcal E(\mathcal G)$,}
\end{equation}
and 
\begin{equation}
\label{eqn:toda2}
\sum_j \frac{\alpha_{i,m} - \alpha_{j,m}}{z_i -z_j} = 0\quad \text{for all vertices $i \in \mathcal V(\mathcal G)$,}
\end{equation}
where the index $j$ runs over all the edges $(ij)$ incident to the vertex $i$ and $m$ is the left face to $e_{ij}$.
\end{enumerate}
\end{Definition}
It is easy to see that for a double graph $\mathcal D$ given by 
 $\mathcal V(\mathcal D) = 
 \mathcal V(\mathcal G) 
 \cup \mathcal V(\mathcal G^*)$, a labelling 
 $\alpha : \mathcal E (\mathcal D) \to \C$ gives a corner $\mathcal C (\mathcal G) \to \C$ which satisfies \eqref{eqn:toda1}. Thus a cross-ratio system \eqref{eq:def-cross} on 
 $\mathcal D$  gives an additive 
 rational Toda system  \eqref{eqn:toda2} on $\mathcal G$.
 We remark that the cross-ratio system on $\mathcal D$ also gives 
 an additive 
 rational Toda system on $\mathcal G^*$ (the so-called dual additive 
 rational Toda system), see Section \ref{sc:cross-ratio}.

\begin{figure}[ht!]
\def\svgwidth{0.4\textwidth}
\begingroup%
  \makeatletter%
  \providecommand\color[2][]{%
    \errmessage{(Inkscape) Color is used for the text in Inkscape, but the package 'color.sty' is not loaded}%
    \renewcommand\color[2][]{}%
  }%
  \providecommand\transparent[1]{%
    \errmessage{(Inkscape) Transparency is used (non-zero) for the text in Inkscape, but the package 'transparent.sty' is not loaded}%
    \renewcommand\transparent[1]{}%
  }%
  \providecommand\rotatebox[2]{#2}%
  \newcommand*\fsize{\dimexpr\f@size pt\relax}%
  \newcommand*\lineheight[1]{\fontsize{\fsize}{#1\fsize}\selectfont}%
  \ifx\svgwidth\undefined%
    \setlength{\unitlength}{326.07024813bp}%
    \ifx\svgscale\undefined%
      \relax%
    \else%
      \setlength{\unitlength}{\unitlength * \real{\svgscale}}%
    \fi%
  \else%
    \setlength{\unitlength}{\svgwidth}%
  \fi%
  \global\let\svgwidth\undefined%
  \global\let\svgscale\undefined%
  \makeatother%
  \begin{picture}(1,0.7999957)%
    \lineheight{1}%
    \setlength\tabcolsep{0pt}%
    \put(0,0){\includegraphics[width=\unitlength,page=1]{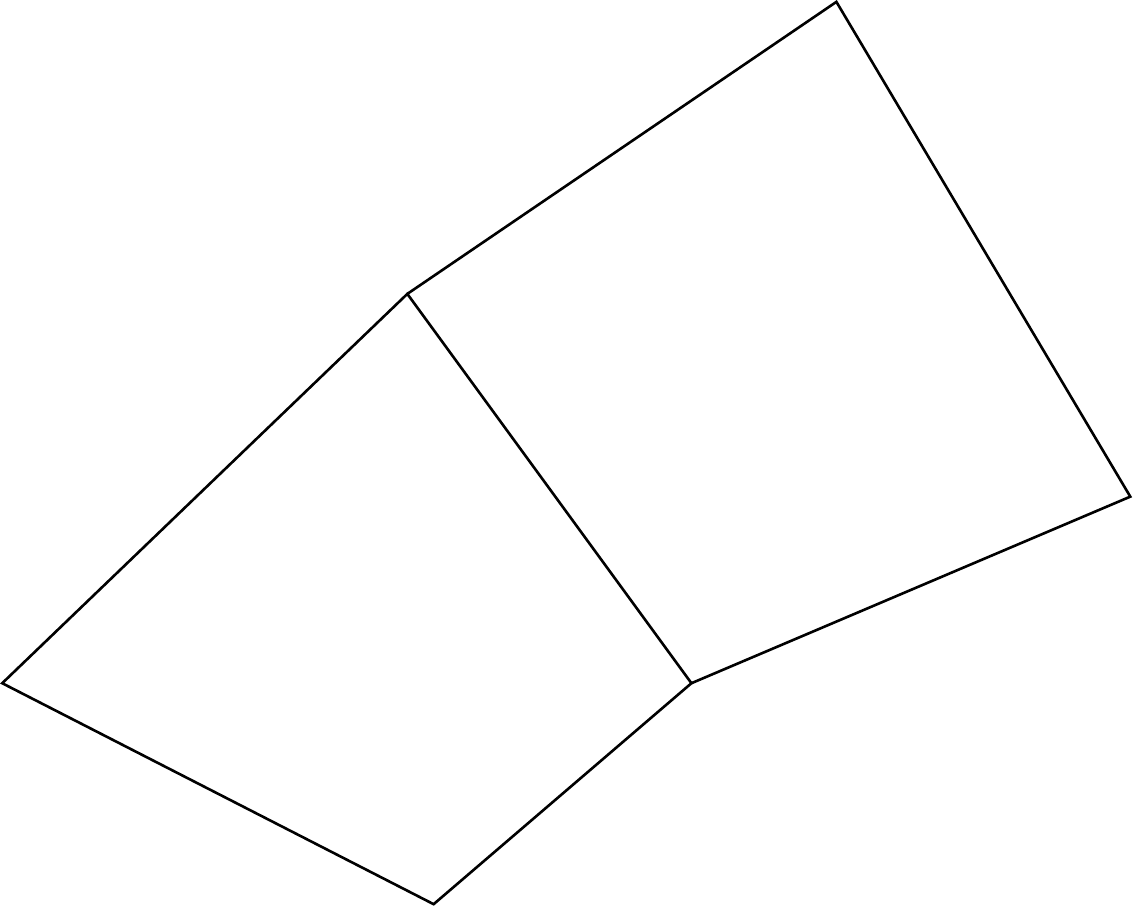}}%
    \put(0.28748469,0.56999331){\color[rgb]{0,0,0}\makebox(0,0)[lt]{\lineheight{1.25}\smash{\begin{tabular}[t]{l}$z_i$\end{tabular}}}}%
    \put(0.65550356,0.12311333){\color[rgb]{0,0,0}\makebox(0,0)[lt]{\lineheight{1.25}\smash{\begin{tabular}[t]{l}$z_j$\end{tabular}}}}%
    \put(0.25962809,0.2082657){\color[rgb]{0,0,0}\makebox(0,0)[lt]{\lineheight{1.25}\smash{\begin{tabular}[t]{l}$m$\end{tabular}}}}%
    \put(0.64367436,0.46944538){\color[rgb]{0,0,0}\makebox(0,0)[lt]{\lineheight{1.25}\smash{\begin{tabular}[t]{l}$n$\end{tabular}}}}%
    \put(0,0){\includegraphics[width=\unitlength,page=2]{toda.pdf}}%
    \put(0.3016874,0.38725814){\color[rgb]{0,0,0}\makebox(0,0)[lt]{\lineheight{1.25}\smash{\begin{tabular}[t]{l}$\alpha_{i,m}$\end{tabular}}}}%
    \put(0.60141501,0.30551661){\color[rgb]{0,0,0}\makebox(0,0)[lt]{\lineheight{1.25}\smash{\begin{tabular}[t]{l}$\alpha_{j,n}$\end{tabular}}}}%
  \end{picture}%
\endgroup%

\caption{Additive rational Toda system.}\label{fig:toda}
\end{figure}

\begin{Remark}
The additive rational Toda system induces the discrete  holomorphic quadratic differential  \cite{Lam2016} in the following way. Recall that the holomorphic quadratic differential is defined by a map $q:\mathcal E(\mathcal G) \rightarrow \mathbb{C}$ such that
\begin{gather}
\label{eqn:hqd}
\sum_j q_{ij} = 0, \quad \text{for all vertex $i \in \mathcal{V}(\mathcal G)$},\\
\sum_j \frac{q_{ij}}{z_i-z_j} =0, \quad \text{for all vertex $i \in \mathcal{V}(\mathcal G)$},\label{eqn:hqd2}
\end{gather}
where $j$ runs through all the neighbouring vertices of $i$. Suppose given an additive rational Toda system.  Let $q_{ij} := \alpha_{i,m}-\alpha_{j,m}$ where $m$ is the left face of the edge $e_{ij}$. By \eqref{eqn:toda1} we have $q_{ij} = q_{ji}$ and $q$ well-defined at the un-oriented edge. Then \eqref{eqn:hqd2} is clearly satisfied.  A simple calculation shows that it satisfies \eqref{eqn:hqd}. Besides, the induced holomorphic quadratic differential satisfies one more additional condition
\[\sum_{ij} q_{ij} = 0, \quad \text{for all face $k\in \mathcal{F}(\mathcal G)$},\]
where $(ij)$ runs through all the edges around the face $k$.
\end{Remark}

 The other main result in this paper 
 is a representation formula for 
 isothermic constant mean curvature surfaces for \textit{general graphs} 
 in terms of solutions (the so-called normalized potentials) 
 of additive Rational Toda systems.
\begin{Theorem}[The Weierstrass type representation for isothermic constant mean 
 curvature surfaces]
 Every solution of an additive rational Toda system on 
 a general graph $\mathcal G$
 gives an isothermic constant mean curvature surface.
 Conversely, the extended frame of an isothermic constant mean 
 curvature surface on $\mathcal G$ 
 gives a solution of an additive rational Toda system.
\end{Theorem}
 The precise statements and proofs will be given in Theorem.~\ref{thm:extendedgeneral} and Theorem.~\ref{thm:cmc_general}.

\begin{Remark}
 Here the extended frame for general isothermic constant mean curvature surfaces has not been known, and thus we will give a definition for it 
 in Section~\ref{sc:Iso} by using the Weierstrass type representation.
  Note that the extended frames will naturally induce 
 isothermic constant mean curvature surfaces by the Sym-Bobenko formula.
\end{Remark}
\section{Discrete constant mean curvature surfaces on quadrilateral nets}\label{sc:isopara}

 It is known \cite{bobenko1999discretization,bobenko1996isothermic} that  discrete isothermic \textit{parametrized}
 constant mean curvature surfaces  are defined on 
 $\Z^2$ and obtained from the extended frames. In this section we extend 
 the basic construction of the extended frames for constant mean curvature 
 surfaces on even quad-graphs $\mathcal D$ 
 and generalize them using the 
 mean curvature $H$ and the Hopf differential $\alpha$, which 
 takes values in $\C^{\times}$.

\subsection{The extended frames and construction of constant mean 
 curvature  surfaces}
\label{sbsc:extended}
 Recall that the extended frame of a discrete isothermic parametrized 
 CMC surfaces on $\mathbb Z^2$ from \cite[(4.14)]{bobenko1999discretization}: The extended frame $\Phi_{n, m}$ is given by the transition matrices $U$, $V$ defined by 
 $\Phi_{n+1, m} =  U\Phi_{n, m}$ and $\Phi_{n, m+1}  = V \Phi_{n, m}$, where
\begin{equation}\label{eq:UV}
\left\{
\begin{array}{l}
\displaystyle
 U = \frac{1}{p} \begin{pmatrix} a & - \l u - \l^{-1} u^{-1}  \\
\l u^{-1} + \l^{-1} u & \bar a \\ \end{pmatrix} 
\\[0.3cm]
\displaystyle
 V = \frac{1}{q} \begin{pmatrix} b & - i \l v + i \l^{-1} v^{-1}  \\
i \l v^{-1} - i \l^{-1} v & \bar b \\ \end{pmatrix}
\end{array}
\right.,
\end{equation}
 with 
\[
 p^2=\det U= \l^2 + \l^{-2} +|a|^2 + u^2 + u^{-2}, \quad
 q^2= \det V=-\l^2 - \l^{-2} +|b|^2 + v^2 + v^{-2}.
 \]
 Here $a, b$ take values in $\C$ and $u, v$ take values in $\R^{\times}$.
 The important \textit{ansatz} about $p$ and $q$ is as follows: 
\begin{equation}\label{eq:ansatz}
\mbox{ The zeros of $p$ and $q$ with respect to $\l$ are 
 $m$- and $n$-independent, respectively.}
\end{equation}
 Denote square of the zeros of $p$ by $\alpha, \alpha^{-1}$ and square of 
 the zeros of $q$ 
 by $\beta, \beta^{-1}$. It is easy to see that from the form of $p$ and 
 $q$, $\alpha$ is negative and $\beta$ is positive 
 and thus the zeros of $p$ are always \textit{pure imaginary} valued 
 and the zeros of $q$ are always \textit{real} valued. Moreover, it always holds
 \begin{equation}\label{eq:negativecond}
  \frac{\alpha}{\beta}<0.
 \end{equation}
 We then rephrase
 $p^2$ and $q^2$ by 
\[
 p^2 =- \alpha^{-1} (1 - \l^2 \alpha) (1 - \l^{-2} \alpha), \quad
 q^2 = \beta^{-1} (1 - \l^2 \beta) (1 - \l^{-2} \beta).
\]
 The ansatz \eqref{eq:ansatz} for $\alpha$ and $\beta$ can be understood as 
 a labelling, i.e., $\alpha$ satisfies the condition $(1)$ 
 in Definition \ref{def:cross_ratio}.

 We now generalize the formulation of 
 the extended frame \eqref{eq:UV} to arbitrary even quad graph $\mathcal D$
 with the constant mean curvature parameter $H \in \R$ 
 and the Hopf differential $\alpha :\mathcal D \to \C^{\times}$ which is 
 a labelling. Note that $\alpha$ is not necessary real-valued and 
 we do not assume the condition 
 \eqref{eq:negativecond}.

\begin{Definition}\label{def:extended}
 The \textit{extended frame} $\Phi= \Phi(v_i, \l)$ is 
 defined to be 
 a map from $\mathcal V ( \mathcal{D})$ 
 into the loop group $\LSU$, i.e., 
 a set of maps from the unit circle $S^{1}$ of the complex plane 
 into  the unitary group of degree two $\SU$
 such that 
 it satisfies  the following relation
 for each edge $\mathfrak e =(v_i, v_j)\in \mathcal E(\mathcal D)$:
\begin{gather}\label{eq:dextended}
 \Phi(v_j, \l) =\frac{1}{\kappa(\mathfrak e, \l)} U(\mathfrak e, \l) \Phi(v_i, \l), \\
\intertext{where}
\label{eq:dextendedU}
 U(\mathfrak{e}, \l) = 
\begin{pmatrix}
 d & \l^{-1} H u - \l \bar \alpha \bar u^{-1}  \\
\l^{-1} \alpha u^{-1}- \l H \bar u  & \bar d
\end{pmatrix},  
\intertext{and}
 \kappa(\mathfrak{e}, \l) = \sqrt{\det U(\mathfrak{e}, \l)} =
 \sqrt{(1 - \l^{-2} H \alpha)(1 - \l^2 H \bar \alpha)}.
 \label{eq:kappa}
\end{gather}
 Here $u = u (\mathfrak{e})$ and $d= d(\mathfrak{e})$ 
 are assumed to take values in $\C$, and 
 $\alpha = \alpha(\mathfrak{e})$ is assumed to be a labelling on $\mathcal D$
 which takes values in $\C^{\times}$ and satisfies $|\sqrt{H \alpha}| \neq 1$
 and $\arg \alpha_1 \neq \arg \alpha_2$, where 
 $\alpha_1 = \alpha_{01} = \alpha_{32}$ and $\alpha_2 = \alpha_{03} = \alpha_{12}$,
 on any elementary quadrilateral $(v_0, v_1, v_2, v_3)$.
 The function $\alpha$ will be called the \textit{Hopf differential} and the real constant parameter $H$ will be called the \textit{mean curvature}.
\end{Definition}
 The precise definition of the loop group $\LSU$ can be found in Appendix \ref{app:loopgroups}.
\begin{Remark}\label{rm:kappa}
\mbox{}
\begin{enumerate}
\item The assumption $\sqrt{H \alpha} \neq 1$ is necessary, since $\Phi$ 
 needs to be an element in $\LSU$. If we use the $r$-loop group $\Lambda^r \mathrm{SU}_{2, \sigma}$ 
 with suitable $1\geq r>0$, then this assumption can be removed. The assumption 
 $\arg \alpha_1 \neq \arg \alpha_2$ is necessary to normalize the determinant of $U$ 
 having the form in \eqref{eq:kappa}, see $(2)$.

\item  The form of $\kappa$ in \eqref{eq:kappa} is not a restriction but
 it can be  automatically satisfied. In fact when $H=0$, then it is easily normalized as in \eqref{eq:kappa}
 by scaling of $U$ with $1/\sqrt{\det U}= 1/\sqrt{|d|^2 + |\alpha|^2 |u|^{-2}}$, 
 and  when $H \neq 0$, the determinant of $U$ can be computed as
\begin{align*}
 \det U &= |d|^2 + H^2 |u|^2 + |\alpha|^2 |u|^{-2} - \lambda^{-2} H \alpha 
- \lambda^
{-2} H \bar \alpha  \\
 &= H |\alpha| \left(p - \lambda^{-2} e^{\sqrt{-1} \arg \alpha}
- \lambda^{-2} e^{-\sqrt{-1} \arg \alpha}\right),
\end{align*}
 where $p = H^{-1}|\alpha|^{-1} |d|^2 + H |\alpha|^{-1}|u|^2 + H^{-1}|\alpha| |u|^{-2}$.
 Then it is easy to see that $p \geq 2$ or $p \leq -2 $, and the labelling property of 
 $\alpha$,  $\arg \alpha_1 \neq \arg \alpha_2$ and the 
 compatibility condition of $U$ imply that the zeros of $\det U$ give also 
 a labelling, i.e., $\det U$ is a labelling.

 Then by scaling a suitable real scalar function $q(\mathfrak e)$ which is a labelling  on 
 $U(\mathfrak e, \l)$, we can assume without loss of generality that  
 the zeros of $\kappa$ are $\alpha$, i.e., $\kappa$ has the form in \eqref{eq:kappa}.
 More precisely, take $q$ as
 a solution of $H^2 |\alpha|^2 q^4 -  (|d|^2 + H^2 |u|^2 + |\alpha|^2 |u|^{-2})q^2 +1 =0$, and define $\tilde \alpha$ as $\tilde \alpha = q^2 \alpha$, then $\tilde 
 \alpha$ is a zero of $q \kappa$. Note that by the above discussion, it is 
 easy to see that  $q$ is a labelling and thus $\tilde \alpha$ 
 is also a labelling. Note that the scaled matrix  $\tilde U = q U$ 
 has the same form as in 
 \eqref{eq:dextendedU}.
\end{enumerate}

\end{Remark}

\mbox{}
\
 From the discrete extended frame we will construct a discrete 
 CMC surface 
 by the Sym-Bobenko formula. 
\begin{Theorem}\label{thm:extendedandCMC}
 Let $H$ be some non-zero constant and $\Phi$ be a the extended frame in \eqref{eq:dextended}. Moreover let 
 $\l =e^{\sqrt{-1}t}, \sigma_3 = \di(1, -1)$. 
 Define two maps into $\operatorname{Im} \mathbb H \cong \R^3$ as follows$:$
\begin{equation} \label{eq:discretecmc}
\left\{
\begin{array}{ll}
\displaystyle  \f &= \displaystyle-\frac{1}{H}\left( \Phi^{-1} \partial_{t} \Phi - \frac{\sqrt{-1}}{2}\ad \Phi^{-1} (\sigma_3) \right)\Big|_{t \in \R}, \\[0.3cm]
\displaystyle   \f^*  &= \displaystyle-\frac{1}{H}\left( \Phi^{-1} \partial_{t} \Phi   + \frac{\sqrt{-1}}{2}\ad \Phi^{-1} (\sigma_3) \right)\Big|_{t \in \R}.
\end{array}
\right.
\end{equation} 
 Then the following statements hold$:$
\begin{enumerate}
\item\label{en:BPclassical} If  $\alpha$ is a real valued and $t = 0$, then
  $\f$ and $\f^*$  are respectively a discrete isothermic parametrized CMC surface and  its Christoffel dual isothermic parametrized CMC surface.
\item\label{en:anyCMC} For any non-zero complex valued labelling $\alpha$ and any $t \in \R$ and 
 a decomposition $\mathcal V(\mathcal D) = \mathcal V(\mathcal G) \cup \mathcal V(\mathcal G^*)$, 
 define two maps 
 $f= \f|_{\mathcal V(\mathcal G)}$ and $f^* = \f^*|_{\mathcal V(\mathcal G^*)}$. Then $f$ and $f^*$ 
 are respectively a discrete CMC surface and its Christoffel dual CMC surface in the sense of Definition {\rm \ref{def:constantmean}}.

\end{enumerate}
\end{Theorem}
 The statement \eqref{en:BPclassical} is the result in \cite{bobenko1999discretization} with $\mathcal D = \mathbb Z^2$ and a slight modification on notation. We will prove it in Appendix 
 \ref{subsc:iso}. 

 To prove \eqref{en:anyCMC}, we need the following Lemmata.
\begin{Lemma}
 Let $\mathfrak f$ be a map defined by the Sym-Bobenko formula \eqref{eq:discretecmc}.
 Define an edge $\mathfrak E_{ij} = \mathfrak R_{ij} + d \mathfrak f_{ij}
 : \mathcal E (\mathcal D) \to \mathbb H$ by 
 \begin{equation}\label{eq:normaltransport}
 \mathfrak R_{ij}  := \frac1{2 H}\partial_t (\log \det U_{ij} )|_{t \in \R} \in \R.
\end{equation}
 Then $\mathfrak E_{ij}$ is the {\rm normal transport vector}, i.e.,
\begin{equation*}
\mathfrak E_{ij}^{-1} \cdot \mathfrak n_j \cdot \mathfrak E_{ij}  = -\mathfrak n_i,
\end{equation*}
 where $\mathfrak n$ is the unit normal to a surface $\mathfrak f$ given 
 by $\mathfrak n = \frac{\sqrt{-1}}2 \ad \Phi^{-1} (\sigma_3)|_{t \in \R}$.
\end{Lemma}
\begin{proof}
 The edge $\mathfrak E_{ij}$ can be rephrased as
\begin{align}
\label{eq:normaltransport2} \mathfrak E_{ij} 
  &= \frac1{2 H}\partial_t (\log \det U_{ij} ) + \mathfrak f_j - \mathfrak f_i  \\
\nonumber  &= - \frac1{H} \Phi_j^{-1} 
 \left(-\left(\partial_t \kappa_{ij}^{-1}\right) U_{ij}+ \partial_t \left(\kappa_{ij}^{-1} U_{ij} \right)- 
 \frac{\sqrt{-1}}{2} [\sigma_3, \kappa_{ij}^{-1}U_{ij}] \right) \Phi_i   \\
\nonumber  &= - \frac1{\kappa_{ij} H} \Phi_j^{-1}
 \left( \partial_t U_{ij}- 
 \frac{\sqrt{-1}}{2} [\sigma_3, U_{ij}] \right) \Phi_i.   
\end{align}
 Since $\partial_t U_{ij}-\frac{\sqrt{-1}}{2} [\sigma_3, U_{ij}]$ is an off-diagonal 
 matrix, the claim $\mathfrak E_{ij}^{-1} \cdot \mathfrak n_j \cdot \mathfrak E_{ij}  = -\mathfrak n_i$ easily follows.
\end{proof}
\begin{Remark}
\mbox{}
\begin{enumerate}
\item From \eqref{eq:normaltransport}, it is clear that $\mathfrak R_{ij} 
 = \mathfrak R_{ji}$ and  on  any elementary quadrilateral $(v_0, v_1, v_2, v_3)$, 
 $\mathfrak R_{01} = \mathfrak R_{32}$ and  $\mathfrak R_{03} = \mathfrak R_{12}$
 from the labelling property.
\item A similar statement holds for the map $\mathfrak f^*$, i.e., there exists 
 an edge $\mathfrak E_{ij}^* = \mathfrak R^*_{ij} + d \mathfrak f_{ij}^*$ such 
 that $(\mathfrak E_{ij}^*)^{-1} \cdot \mathfrak n_j \cdot \mathfrak E_{ij}^*  
 = -\mathfrak n_i$ holds. In fact,  one can choose 
 $\mathfrak R^*_{ij} =  \mathfrak R_{ij}$ and 
\[
\mathfrak E_{ij}^* = - \frac1{\kappa_{ij} H} \Phi_j^{-1}
 \left( \partial_t U_{ij}+
 \frac{\sqrt{-1}}{2} [\sigma_3, U_{ij}] \right) \Phi_i
\]
holds.
\end{enumerate}
\end{Remark}
 Defining that 
 $f= \mathfrak f|_{\mathcal V(\mathcal G)}$ and  
 $f^*= \mathfrak f^*|_{\mathcal V(\mathcal G^*)}$, 
we build the diagonal hyperedges as 
\begin{align*}
{E}_{02}:= R_{02} + df_{02},  \quad {E}^*_{31}  := R_{31} + df^*_{31},
\end{align*}
where $R_{02}=-\mathfrak R_{30}+\mathfrak R_{32} $
 and $R_{31}=\mathfrak R_{30}+\mathfrak R_{01}$.
We define the hyperedges across the primal and dual nets by:
\begin{gather}
\label{eqn:cross_hyper1}
 E_{32} := R_{32} + f_2 -f_3^*,\quad 
 E_{30} := R_{30} + f_0 - f_3^*, \\
\label{eqn:cross_hyper2}
 E_{01} := R_{01} + f_1^* - f_0, \quad 
 E_{21} := R_{21} + f_1^* - f_2.
\end{gather}
 where $R_{ij} = \mathfrak R_{ij}$ for $(ij) \in \{(01),(12),(32),(03)\}$.
 Then $E_{02}$ and $E_{31}^*$ can be written as
\[
 E_{02} = -  E_{30}+ E_{32}  = -E_{21}+ E_{01}, \quad E_{31}^* = E_{30} + E_{01}
 = E_{32} + E_{21}.
\]
\begin{Lemma}
\label{lem:cross_closing}
Let $E_{01}$, $E_{12}$, $E_{23}$ and $E_{03}$ be the hyperedges defined 
in \eqref{eqn:cross_hyper1} and \eqref{eqn:cross_hyper2}, respectively.
 Then the following statement holds$:$
\begin{gather}
\label{eqn:cross_closing1}
|E_{32}| = |E_{01}| = \frac{\sqrt{\kappa_1^2 +\Re( \lambda^2 H \alpha_1)}}{\kappa_1 H},  \\ \label{eqn:cross_closing2}
\quad |E_{12}| = |E_{03}| = \frac{\sqrt{\kappa_2^2 +\Re( \lambda^2 H \alpha_2)}}{\kappa_2 H}, \\
\label{eqn:cross_closing3}
E_{12}\cdot E_{01} = E_{32} \cdot E_{03},
\end{gather}
 where $\kappa_1 = \kappa_{01} = \kappa_{32}$ and 
$\kappa_2 = \kappa_{03} = \kappa_{12}$.
Moreover, by using \eqref{eqn:cross_closing1},
 \eqref{eqn:cross_closing3} is equivalent with 
\begin{equation}\label{eqn:cross_closing4}
E_{23} = E_{02}^{-1}\cdot E_{10} \cdot E_{02}.
\end{equation}
 Note that $E_{ji} = \overline{E_{ij}}$ 
 with $(ij) \in \{(01),(12),(32),(03)\}$. 
 Similary 
\begin{equation}\label{eqn:cross_closing5}
E_{32}\cdot E_{03} = E_{12} \cdot E_{01} 
\end{equation}
holds, and by using \eqref{eqn:cross_closing2}, 
 \eqref{eqn:cross_closing5} is equivalent with
\begin{equation}\label{eqn:cross_closing6}
 E_{12} =  E_{02}^{-1} \cdot E_{03} \cdot E_{02}.
\end{equation}
\end{Lemma}
\begin{proof}
We calculate $E_{01}, E_{03}, E_{12}$ and $E_{32}$ by the Sym-Bobenko formula as
\begin{align*}
E_{01} &= R_{01} + f_1^* - f_0 = -\frac{1}{\kappa_1 H}\Phi_1^{-1}\cdot \mathfrak{U}_{01}\cdot \Phi_0, \quad 
E_{03} &= R_{03} + f_3^* - f_0 = -\frac{1}{\kappa_2 H}\Phi_3^{-1}\cdot \mathfrak{U}_{03}\cdot \Phi_0,\\
E_{12} &= R_{12} + f_2 - f_1^*= -\frac{1}{\kappa_2 H}\Phi_2^{-1}\cdot \mathfrak{U}_{12}\cdot \Phi_1,\quad 
E_{32} &= R_{32} + f_2 - f_3^* = -\frac{1}{\kappa_1 H}\Phi_2^{-1}\cdot \mathfrak{U}_{32}\cdot \Phi_3,
\end{align*}
where 
\begin{align*}
\mathfrak{U}_{01} 
&= \partial_t U_{01} + \frac{\sqrt{-1}}{2}\sigma_3 U_{01} 
 + \frac{\sqrt{-1}}{2} U_{01} \sigma_3, 
\quad \mathfrak{U}_{03} 
= \partial_t U_{03} -\frac{\sqrt{-1}}{2}\sigma_3 U_{03} 
 - \frac{\sqrt{-1}}{2} U_{03} \sigma_3, \\
\mathfrak{U}_{12} 
&= \partial_t U_{12} - \frac{\sqrt{-1}}{2}\sigma_3 U_{12} 
 - \frac{\sqrt{-1}}{2} U_{12} \sigma_3, \quad
\mathfrak{U}_{32} 
= \partial_t U_{32} + \frac{\sqrt{-1}}{2}\sigma_3 U_{32} 
 + \frac{\sqrt{-1}}{2} U_{32} \sigma_3.
\end{align*}
Clearly, \eqref{eqn:cross_closing3} is equivalent to 
\[
 \mathfrak{U}_{12}\cdot \mathfrak{U}_{01} = \mathfrak{U}_{32} \cdot \mathfrak{U}_{03}.
\]
which is, by a direct calculation, equivalent to the compatibility condition
\[
U_{12}U_{01}= U_{32}U_{03}.
\]
 Thus \eqref{eqn:cross_closing3} holds.
 We now take the conjugation to \eqref{eqn:cross_closing3}, i.e.,
\[
 \overline{E_{03}} \cdot \overline{E_{32}} = 
  \overline{E_{01}} \cdot \overline{E_{12}}.
\]
 Moreover, using \eqref{eqn:cross_closing1} and
 $E_{02} = E_{32} - \overline{E_{03}} = E_{01} - \overline{E_{12}}$,
 we have 
\[
 E_{02} \cdot \overline{E_{32}} =  \overline{E_{01}} \cdot E_{02}
 \Longleftrightarrow  E_{02} \cdot E_{23} = E_{10} \cdot E_{02}.
\]
Similary \eqref{eqn:cross_closing5} follows and it is equivalent with 
\eqref{eqn:cross_closing6}.
\end{proof}

\begin{proof}[Proof of Theorem $\ref{thm:extendedandCMC}$ \eqref{en:anyCMC}]
 It is easy to see that \eqref{eqn:cross_closing4} and \eqref{eqn:cross_closing6} 
 imply 
\[
f^*_{3} - f_2 = E^{-1}_{02} \cdot (f_0-f^*_1) \cdot E_{02} \quad \mbox{and} \quad
f^*_{3} - f_0 = E^{-1}_{02} \cdot (f_2-f^*_1) \cdot E_{02},
\]
respectively. Thus Definition \ref{def:constantmean} (1) holds.
Let us verify the isothermic condition \eqref{eq:dis_iso} in Definition \ref{def:genIso}:
\begin{align*}
E_{02} \cdot E^*_{31} &= E_{02} \cdot (E_{30} + E_{01}) \\
&= E_{02} \cdot E_{30} + E_{32} \cdot E_{02} \\
&= (E_{32} - E_{30}) \cdot E_{30} +  E_{32} \cdot (E_{32} - E_{30})\\
&= -E_{30}\cdot E_{30} + E_{32} \cdot E_{32}.
\end{align*}
 In the second equality, we use the expression in \eqref{eqn:cross_closing4}. 
By symmetry, two terms above will be canceled with the terms in neighboring parallelogram, see Figure~\ref{fig:neigh_faces}. 
\begin{figure}[t]
\centering
\def\svgwidth{0.4\textwidth}
\begingroup%
  \makeatletter%
  \providecommand\color[2][]{%
    \errmessage{(Inkscape) Color is used for the text in Inkscape, but the package 'color.sty' is not loaded}%
    \renewcommand\color[2][]{}%
  }%
  \providecommand\transparent[1]{%
    \errmessage{(Inkscape) Transparency is used (non-zero) for the text in Inkscape, but the package 'transparent.sty' is not loaded}%
    \renewcommand\transparent[1]{}%
  }%
  \providecommand\rotatebox[2]{#2}%
  \newcommand*\fsize{\dimexpr\f@size pt\relax}%
  \newcommand*\lineheight[1]{\fontsize{\fsize}{#1\fsize}\selectfont}%
  \ifx\svgwidth\undefined%
    \setlength{\unitlength}{349.25760738bp}%
    \ifx\svgscale\undefined%
      \relax%
    \else%
      \setlength{\unitlength}{\unitlength * \real{\svgscale}}%
    \fi%
  \else%
    \setlength{\unitlength}{\svgwidth}%
  \fi%
  \global\let\svgwidth\undefined%
  \global\let\svgscale\undefined%
  \makeatother%
  \begin{picture}(1,0.53786437)%
    \lineheight{1}%
    \setlength\tabcolsep{0pt}%
    \put(0,0){\includegraphics[width=\unitlength,page=1]{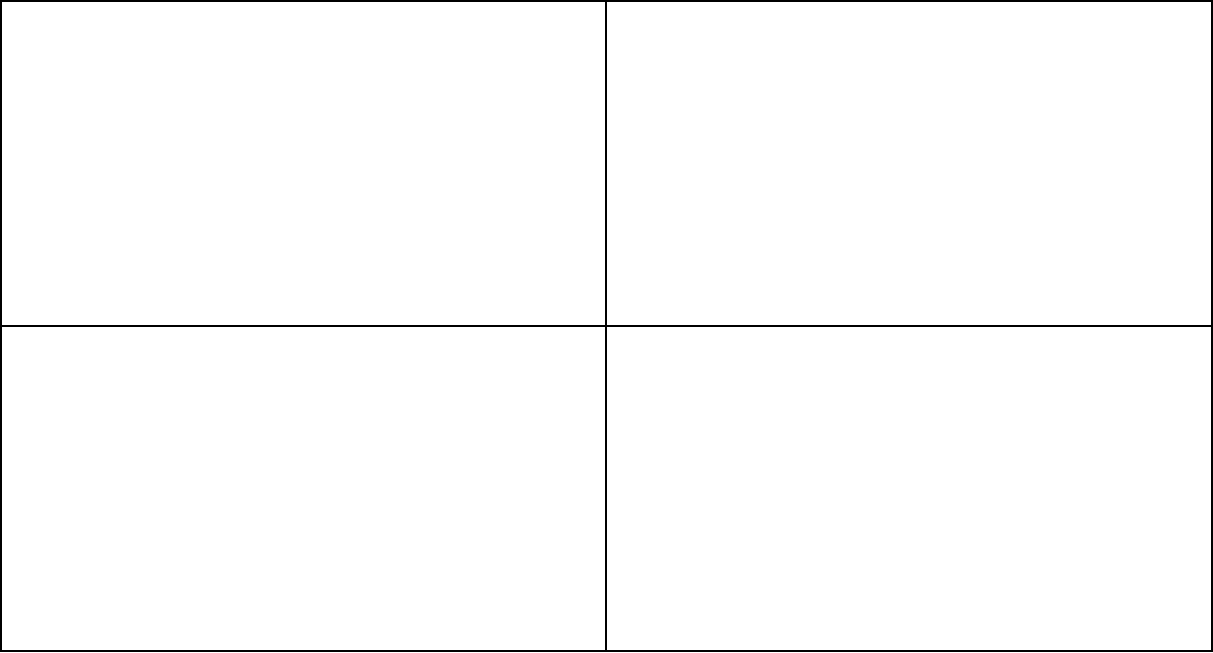}}%
    \put(0.5405349,0.02962715){\color[rgb]{0,0,0}\makebox(0,0)[lt]{\lineheight{1.25}\smash{\begin{tabular}[t]{l}$0$\end{tabular}}}}%
    \put(0.91797803,0.02528873){\color[rgb]{0,0,0}\makebox(0,0)[lt]{\lineheight{1.25}\smash{\begin{tabular}[t]{l}$1$\end{tabular}}}}%
    \put(0.91580867,0.19466178){\color[rgb]{0,0,0}\makebox(0,0)[lt]{\lineheight{1.25}\smash{\begin{tabular}[t]{l}$2$\end{tabular}}}}%
    \put(0.53619629,0.19683096){\color[rgb]{0,0,0}\makebox(0,0)[lt]{\lineheight{1.25}\smash{\begin{tabular}[t]{l}$3$\end{tabular}}}}%
    \put(0,0){\includegraphics[width=\unitlength,page=2]{neigh_faces.pdf}}%
  \end{picture}%
\endgroup%

\caption{}\label{fig:neigh_faces}
\end{figure}
 Thus \eqref{eq:dis_iso} holds. This completes the proof.
\end{proof}

\begin{Remark}
\mbox{}
\begin{enumerate}
\item
 In \cite{bobenko1999discretization}, a map 
\[
\f_K  = -\frac{1}{H} \left( \Phi^{-1}  \partial_{t} \Phi\right)\big|_{t \in \R}
\]
 has  been called the discrete 
\textit{constant positive Gaussian curvature}  surface, since 
 it is a parallel surface of a CMC surfaces $\f$ and $\f^*$ with distance $1/(2 H)$.

\item When the Hopf differential $\alpha$ is real-valued and $t \in \R^{\times}$, $\f$ (and 
 $\f^*$) was not known to satisfy any condition of isothermic parametrized surface, even though they are in fact isothermic surfaces in smooth case. 
\end{enumerate}
\end{Remark}
 We next consider discrete minimal surfaces. 
 Recall that the discrete minimal surface can be generated by the discrete Weierstrass representation, \cite{bobenko1996isothermic}:
 Let $z_i= z(v_i)$ be a solution of the 
 cross-ratio system \eqref{eq:def-cross} and take a 
 dual solution of a cross-ratio system, i.e., $z^*$ is defined by 
\[
z_i^* - z_j^* := \frac{\alpha_{ij}}{z_j- z_i},
\]
and $z_i^*$ satisfies the cross-ratio system
\[
 q(z_0^*, z_1^*, z_2^*, z_3^*)= 
\frac{(z_0^*-z_1^*)(z_2^*-z_3^*)}{(z_1^*-z_2^*)(z_3^*-z_0^*)} = 
\frac{\alpha_2}{\alpha_1}
\]
 on any elementary quadrilateral $(v_0, v_1, v_2, v_3)$, where 
 $\alpha_1 = \alpha_{01} = \alpha_{32}, \alpha_2 = \alpha_{03} = \alpha_{12}$.
 Then a minimal surface is given by
\begin{equation}
\label{eqn:minimaledge}
d\mathfrak{f}_{ij} = \left. \frac{1}{2} \mathrm{Re}
\left\{ \frac{\lambda^{-2} \alpha_{ij}}{z_j^*-z_i^*} \left( 1-z_j^*z_i^*,\sqrt{-1}(1+z_j^*z_i^*),z_j^*+z_i^*\right)
\right\}\right|_{\lambda =1},
\end{equation}
and the normal $\mathfrak{n}:\mathbb{Z}^2 
 \rightarrow S^2 \subset 
 \R^3 \cong \C \times \R$
at vertices is given by
\begin{align}
\label{eq:mini_norm}
(\mathfrak{n}^x_i + \sqrt{-1}\mathfrak{n}^y_i,\mathfrak{n}^z_i) &= \left(\frac{2z_i^*}{1+\lvert z_i^*\rvert^2},\frac{\lvert z_i^*\rvert^2-1}{\lvert z_i^*\rvert^2+1}\right).
\end{align} 
 Here note that the Hopf differential $\alpha$ assumed to be real-valued.

 We now show that the Sym-Bobenko formula in \eqref{eq:discretecmc} converges 
 to the discrete Weierstrass representation when $H$ goes to $0$, and 
 moreover, it defines a discrete minimal surface on a general graph 
 $\mathcal G$.
\begin{Theorem}\label{thm:extendedandminimal}
 Retain the assumptions in Theorem {\rm \ref{thm:extendedandCMC}}.
\mbox{}
\begin{enumerate}
\item When $H$ converges to $0$, the formula 
 \eqref{eq:discretecmc} induces the discrete Weierstrass representation
 \eqref{eqn:minimaledge} and the dual surface 
 $\f^*$ becomes $\mathfrak n$ defined in \eqref{eq:mini_norm}.
\item For any non-zero complex-valued labelling $\alpha$ and any $t \in \R$ and 
 the decomposition  $\mathcal V(\mathcal D) = \mathcal V(\mathcal G) \cup 
 \mathcal V(\mathcal G^*)$, 
 $f= \f|_{\mathcal V(\mathcal G)}$ given by \eqref{eqn:minimaledge}
 is a discrete minimal surface with the dual surface  $f^*= \f^*
|_{\mathcal V(\mathcal G^*)}$ in the 
 sense of Definition {\rm \ref{def:constantmean} (2)}.
\end{enumerate}
\end{Theorem}
\begin{proof}
(1) When $H$ converges to $0$, the compatibility condition of 
 the discrete extended frame $\Phi$ in \eqref{eq:dextended} 
 can be explicitly obtained by the dual solution of a cross-ratio 
 system:
 Define the frame $\Phi$ to be 
\begin{equation}\label{eq:Phizero}
 \Phi_i = \Phi(v_i, \l) = \frac{1}{\sqrt{1 + |z_i^*|^2}} 
 \begin{pmatrix} 1 & \bar z_i^* \l \\ 
 - z_i^* \l^{-1} &1 \end{pmatrix}.
\end{equation}
 Then the transition matrix $U_{ij} = \Phi_j \Phi_i^{-1}$ can be computed as
\begin{align*}
U_{ij} &= \frac{1}{\sqrt{(1+|z_i^*|^2)(1+|z_j^*|^2)}}
\begin{pmatrix}
 1 + \bar z_i^* z_j^* &-(\bar z_i^* - \bar z_j^*) \l^{} \\
(z_i^* -z_j^*) \l^{-1} & 1 +  z_i^* \bar z_j^* 
\end{pmatrix} \\
& = \frac{1}{\sqrt{(1+|z_i^*|^2)(1+|z_j^*|^2)}}
 \begin{pmatrix}
 1 + \bar z_i^* z_j^* &-\frac{\overline{\alpha_{ij}}}{\bar z_j - \bar z_i} \l^{} \\
\frac{\alpha_{ij}}{z_j - z_i}\l^{-1} & 1 +  z_i^* \bar z_j^* 
\end{pmatrix}.
\end{align*}
 Setting 
\begin{equation}
\label{eq:u}
u_{ij} = \frac{\alpha_{ij}}{z_i^*-z_j^*}\sqrt{(1+\lvert z_i^*\rvert^2)(1+\lvert z_j^*\rvert^2)}, 
\end{equation}
 we conclude $\Phi$ in \eqref{eq:Phizero} is 
 the  extended frame $\Phi$ in \eqref{eq:dextended}
 with $H=0$.

 We now compute the convergence of the edge $d \f_{ij}$:
 By \eqref{eq:discretecmc} the edge satisfies
\begin{align}
d\mathfrak{f}_{ij} &= -\frac{1}{H}\left( \Phi^{-1}_j \partial_{t} \Phi_j   - \frac{\sqrt{-1}}{2}\ad \Phi^{-1}_j (\sigma_3) \right)+ \frac{1}{H}\left( \Phi^{-1}_i \partial_{t} \Phi_i   - \frac{\sqrt{-1}}{2}\ad \Phi_i^{-1} (\sigma_3) \right),\label{eq:spin_min} \\
&=-\frac{1}{\kappa_{ij} H}\Phi_j^{-1}\cdot\left(\partial_tU_{ij} - \frac{\sqrt{-1}}{2} \sigma_3 \cdot U_{ij} + \frac{\sqrt{-1}}{2} U_{ij} \cdot \sigma_3\right)\cdot \Phi_i + \frac{1}{2 H \kappa_{ij}} \partial_t (\log \det U_{ij}) \id\nonumber\\
\nonumber &= 2 \frac{\sqrt{-1}}{\kappa_{ij}} \Phi_{j}^{-1} \cdot 
\begin{pmatrix} 0 & u_{ij} \l^{-1}\\ \overline{u_{ij}}\l & 0 \end{pmatrix} \cdot \Phi_i + \frac{\sqrt{-1}}{\kappa_{ij}} \left( \frac{\lambda^{-2} \alpha_{ij}}{1- \lambda^{-2}H \alpha_{ij}} - \frac{\lambda^{2} \overline{\alpha_{ij}}}{1- \lambda^{2} H\bar \alpha_{ij}}\right) \id.
\end{align}
 Now we can take the limit of $d \f_{ij}$ 
 gauged by a matrix $\left(\begin{smallmatrix} \sqrt{\lambda}^{-1} & 0 \\ 0 & \sqrt{\lambda} \end{smallmatrix}\right)$ as $H$ goes $0$, i.e., 
 \[
 \lim_{H \to 0}\left\{\ad \begin{pmatrix} \sqrt{\lambda}^{-1} & 0 \\0  & \sqrt{\lambda}\end{pmatrix} (d \f_{ij})\right\}, 
 \]
 and plugging  
 \eqref{eq:Phizero} and \eqref{eq:u} into \eqref{eq:spin_min}, we obtain \eqref{eqn:minimaledge} by a straightforward computation. 
 Note that $\lim_{H \to 0}  \kappa_{ij}=1$.

 Form \eqref{eq:discretecmc}, we can rephrase 
 the dual net in terms of the primal net and the normal:
\[\f^* = \f - \frac{\sqrt{-1}}{H}\ad \Phi^{-1} (\sigma_3).\]
 The first term of the right-hand side can be ignored since 
 it has finite norm by \eqref{eq:spin_min}. By normalization 
 the second term is exactly \eqref{eq:mini_norm}. Therefore, the 
 dual net can be understood by scaling limit of the unit normal $\mathfrak n$.

(2) 
 From (1) we know that the arguments for the proof of Theorem
\ref{thm:extendedandCMC} also applies for minimal surfaces. Hence, by
using the proof of \eqref{eq:diag_trans} and taking $H \rightarrow 0$ we
have \eqref{eq:diag_trans_min}, since $f_0$ and $f_2$ vanish by the
scaling limit. Furthermore, $R_{02} =
\mathfrak{R}_{30}-\mathfrak{R}_{32}$ also holds for minimal surfaces.
Hence, the sum $\sum R_{ij}$ vanishes, where $ij$ runs through the
diagonal edges enclosing the vertex $v_3$ in Figure
\ref{fig:neigh_faces}, because $\mathfrak{R}_{30}$ and
$\mathfrak{R}_{32}$ are canceled with the terms in the neighbouring
quadrilaterals in Figure \ref{fig:neigh_faces}.
\end{proof}

\section{The Weierstrass representation and additive rational Toda 
 systems}\label{sc:DPW}
 In this section, applying the Birkhoff decomposition
 to the extended frame, we first obtain a discrete normalized 
 potential (a solution of the cross-ratio system), 
 and conversely a normalized potential gives the extended frames 
 on $\mathcal D$.

 Moreover, we 
 decompose the cross-ratio system into the additive rational Toda system 
 and its dual.
 These basic results on quadrilateral net will be used in the construction of 
 constant mean curvature surfaces on \textit{general} 
 graphs in Section \ref{sc:Iso}.

\subsection{Holomorphic potential and the Weierstrass type representation}\label{subsc:normalized}
 From now on, if $H \neq 0$, then 
 we always assume $H \alpha$  (one of the zeros of $\det U$) 
 is sufficiently large such that  $|\sqrt{H \alpha}|>1$. 
 This condition is imposed by 
 a technical reason of  definitions of the loop groups, see 
 the definition of $\LSLN$ in Appendix \ref{app:loopgroups}. 
 In fact, using the $r$-loop group $\LSLr$ (which is 
 defined as a set of map from the $C^r$ circle with $r<1$ instead of $S^1$), we can avoid this assumption, but for simplicity of the presentation, the condition is assumed.

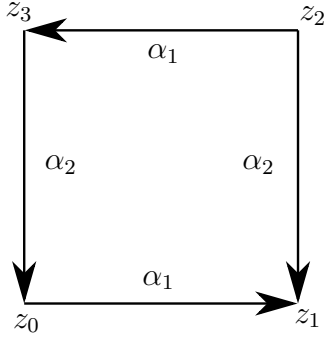
\begin{figure}[ht!]
\def\svgwidth{0.3\textwidth}
\begingroup%
  \makeatletter%
  \providecommand\color[2][]{%
    \errmessage{(Inkscape) Color is used for the text in Inkscape, but the package 'color.sty' is not loaded}%
    \renewcommand\color[2][]{}%
  }%
  \providecommand\transparent[1]{%
    \errmessage{(Inkscape) Transparency is used (non-zero) for the text in Inkscape, but the package 'transparent.sty' is not loaded}%
    \renewcommand\transparent[1]{}%
  }%
  \providecommand\rotatebox[2]{#2}%
  \newcommand*\fsize{\dimexpr\f@size pt\relax}%
  \newcommand*\lineheight[1]{\fontsize{\fsize}{#1\fsize}\selectfont}%
  \ifx\svgwidth\undefined%
    \setlength{\unitlength}{220.57469925bp}%
    \ifx\svgscale\undefined%
      \relax%
    \else%
      \setlength{\unitlength}{\unitlength * \real{\svgscale}}%
    \fi%
  \else%
    \setlength{\unitlength}{\svgwidth}%
  \fi%
  \global\let\svgwidth\undefined%
  \global\let\svgscale\undefined%
  \makeatother%
  \begin{picture}(1,0.90415677)%
    \lineheight{1}%
    \setlength\tabcolsep{0pt}%
    \put(0,0){\includegraphics[width=\unitlength,page=1]{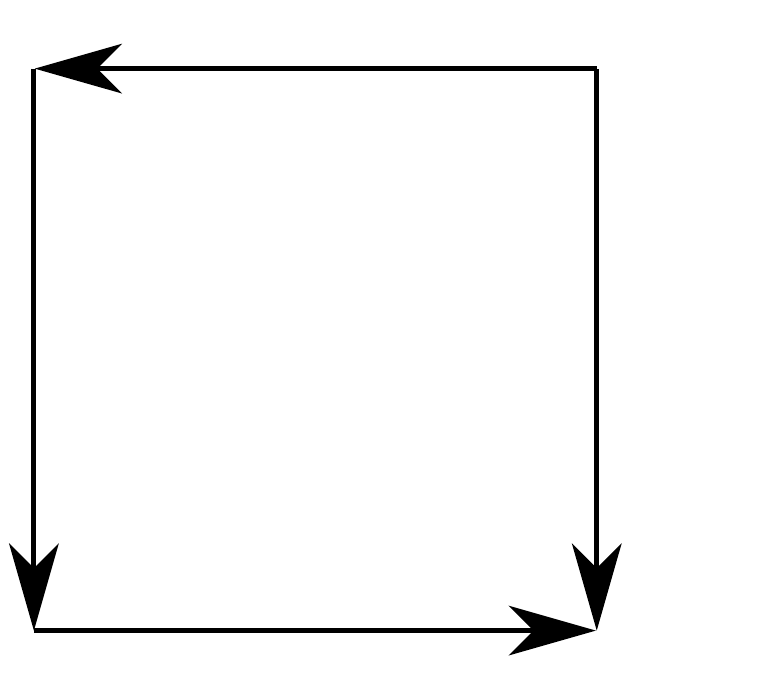}}%
    \put(0.01475649,0.01603809){\color[rgb]{0,0,0}\makebox(0,0)[lt]{\lineheight{1.25}\smash{\begin{tabular}[t]{l}$z_0$\end{tabular}}}}%
    \put(0.76960285,0.02963904){\color[rgb]{0,0,0}\makebox(0,0)[lt]{\lineheight{1.25}\smash{\begin{tabular}[t]{l}$z_1$\end{tabular}}}}%
    \put(0.7832036,0.8422889){\color[rgb]{0,0,0}\makebox(0,0)[lt]{\lineheight{1.25}\smash{\begin{tabular}[t]{l}$z_2$\end{tabular}}}}%
    \put(-0.00564488,0.85248954){\color[rgb]{0,0,0}\makebox(0,0)[lt]{\lineheight{1.25}\smash{\begin{tabular}[t]{l}$z_3$\end{tabular}}}}%
    \put(0.36157772,0.1316453){\color[rgb]{0,0,0}\makebox(0,0)[lt]{\lineheight{1.25}\smash{\begin{tabular}[t]{l}$\alpha_1$\end{tabular}}}}%
    \put(0.37517856,0.74028264){\color[rgb]{0,0,0}\makebox(0,0)[lt]{\lineheight{1.25}\smash{\begin{tabular}[t]{l}$\alpha_1$\end{tabular}}}}%
    \put(0.09976164,0.44106424){\color[rgb]{0,0,0}\makebox(0,0)[lt]{\lineheight{1.25}\smash{\begin{tabular}[t]{l}$\alpha_2$\end{tabular}}}}%
    \put(0.63019411,0.44106424){\color[rgb]{0,0,0}\makebox(0,0)[lt]{\lineheight{1.25}\smash{\begin{tabular}[t]{l}$\alpha_2$\end{tabular}}}}%
  \end{picture}%
\endgroup%

\caption{Cross-ratio system on a quadrilateral}\label{fig:cr}
\end{figure}
Following \cite{Bobenko2002quad} one can rephrase 
 the cross-ratio system with the following matrix form.
\begin{Definition}
\label{def:holo_pot}
 Let $\Phi_-: \mathcal V (\mathcal D)\to \LSLNN$. We call respectively  
 $\Phi_-$  and $L_-$ the \textit{wave function} and the
 \textit{holomorphic potential} if they satisfy the following equation:
\begin{gather}\label{eq:potentialU}
  \Phi_-(v_j, \l) = L_- (\mathfrak e, \l)  \Phi_-(v_i, \l),  
\end{gather}
where 
\begin{align}\label{eq:Lminus}
 L_- (\mathfrak e, \l) & 
  = \dfrac{1}{\sqrt{1 - H \alpha(\mathfrak e) \l^{-2}}}
\begin{pmatrix} 
 1 &  H (z_i-z_{j})\l^{-1} \\
\dfrac{\alpha(\mathfrak e)} {z_i-z_j} \l^{-1}  &  1
\end{pmatrix}, 
\end{align}
 and $\alpha
 = \alpha (\mathfrak e)$ is the Hopf differential.
\end{Definition}
 It is straightforward to verify that the compatibility 
 condition of $L_-$ in Definition \ref{def:holo_pot} is the cross-ratio 
 system for $z_i=z(v_i)$ in Definition \ref{def:cross_ratio} 
 and vice versa.
 Thus Definitions \ref{def:cross_ratio} 
 and \ref{def:holo_pot} are equivalent.
 From the following theorem, we clearly see the reason that
 $L_{-}$ has been called  the holomorphic potential.
\begin{Theorem}\label{thm:normalized}
 Let $\Phi$ be the extended frame as in \eqref{eq:dextended}.
 Perform the 
 Birkhoff decomposition of {\rm Theorem  \ref{Thm:Birkhoff}} to 
 $\Phi$ as 
\begin{equation}\label{eq:BirkhoffofF}
\Phi = \Phi_{+} \Phi_{-},
\end{equation}
 i.e., $\Phi_-: \mathcal V (\mathcal D)\to \LSLNN$ 
 and $\Phi_+: \mathcal V (\mathcal D)\to \LSLP$. Then $\Phi_-$ and 
 $L_-$ are respectively the wave function and the normalized potential.
\end{Theorem}
 The proof will be given in Appendix \ref{subsec:birk_quad}.

 We now consider the converse construction, which is 
 commonly called the DPW representaition,  
 based on another loop group decomposition, the 
 Iwasawa decomposition. 
 More precisely, we will show a construction of the 
 extended frame from a holomorphic potential.

 Let us start from a discrete holomorphic function 
 $z: \mathcal V (\mathcal D) \to 
 \C, z_i = z(v_i)$,  as a solution of 
 the cross-ratio system  \eqref{eq:def-cross} 
 with a given labelling $\alpha=\alpha(\mathfrak e) \in \C^{\times}$, and 
 define a holomorphic potential $L_-$ as in 
 \eqref{eq:Lminus}.
 Moreover,  let $\Phi_-(v_i,\lambda)$ be the wave function from $\mathcal{V}(\mathcal{D})$ into the loop group  $ \LSLN$, such that for each edge $\mathfrak{e} = (v_i , v_j)$, $\Phi_-$ satisfies
\begin{gather}\label{eq:potentialU2}
  \Phi_-(v_j, \l) = L_- (\mathfrak e, \l)  \Phi_-(v_i, \l).
\end{gather}
 Then the following theorem holds.
\begin{Theorem}[The Weierstrass type representation]\label{thm:DPW}
 Retain the notation above. Perform the 
 Iwasawa decomposition of {\rm Theorem \ref{Thm:Iwasawa}} to 
 $\Phi_-$ as 
\begin{equation}\label{eq:IwasawaforFm}
\Phi_-  = \Phi_+ \Phi, 
\end{equation}
 i.e., $\Phi: \mathcal V (\mathcal D) \to \LSU$ 
 and $\Phi_+: \mathcal V (\mathcal D) \to \LSLP$. Then $\Phi$ is
 the extended frame in \eqref{eq:dextended}. Furthermore,
 by using Theorem {\rm \ref{thm:extendedandCMC}}, 
 a discrete isothermic CMC surface  is obtained from 
 the normalized potential.
\end{Theorem}
 The proof will be given in Appendix \ref{subsec:birk_quad}.

\subsection{Holomorphic potential for additive rational Toda system}\label{sc:cross-ratio}
 The cross-ratio system induces a dual solution 
 on $\mathcal D$, which is called the \textit{dual} discrete holomorphic potential as follows:
 Let $z_i= z(v_i)$ be a solution of the cross-ratio system 
 \eqref{eq:def-cross}.
 Define a function $z_i^* = z^*(v_i)$ as
\begin{equation}
\label{eq:dual_cr}
H(z_i^* - z_j^*) = \frac{\alpha (\mathfrak e)}{ z_j- z_i}.
\end{equation}
 Then $z^*: \mathcal V (\mathcal D) \to \C$ satisfies the cross-ratio system: 
 \begin{equation}\label{def:dual}
  q(z_0^*,z_1^*, z_2^*, z_3^*) =  \frac{(z_0^* -z_1^*)(z_2^*-z_3^*)}{(z_1^* -z_2^*)(z_3^* - z_0^*)} = \frac{\alpha_2}{\alpha_1}.
 \end{equation}
 Thus $z^*$ is also a discrete holomorphic function and called a 
 \textit{dual discrete holomorphic function}. 
 Moreover, the discrete normalized potential $L_-$ defined in 
 \eqref{eq:Lminus} can be rephrased as
\begin{align}\label{eq:Lminusdual}
 L_-^* (\mathfrak e^*, \l) &= \dfrac{1}{\sqrt{1 - H \alpha(\mathfrak e^*) \l^{-2}}}
\begin{pmatrix} 
 1 &  \dfrac{\alpha(\mathfrak e^*)} {z_j^*-z_i^*}\l^{-1} \\
H (z_j^*-z_{i}^*) \l^{-1}  &  1
\end{pmatrix}. 
\end{align}
 Note that $ \alpha(\mathfrak e^*) =  \alpha(\mathfrak e)$.
 The matrix $ L_-^* (\mathfrak e, \l)$  is called 
 the \textit{dual holomorphic potential}.
 Accordingly, the \textit{dual wave function} $\Phi_-^*$ can be defined as
 \begin{equation}\label{eq:potentialUdual}
\Phi_-^*(v_i, \l) = L_-^*(\mathfrak e^*, \l) \Phi_-^*(v_j, \l).
 \end{equation}

\begin{Remark}
 It is easy to see that $z^*$ actually gives a solution to the same system,
 see \cite[Proposition 12]{Bobenko2002quad}. The choices of $z_i$ and $z_i^*$ are unique up to the initial conditions.
\end{Remark}
As we discussed in Section \ref{sbsc:Toda}, the cross-ratio 
system induces 
the addtive rational Toda system on $\mathcal G$. Note that the dual cross-ratio system \eqref{eq:dual_cr} also induces the \textit{dual additive rational Toda system} on $\mathcal{G}^*$. Thus 
 the cross-ratio system induces the pair of additive rational Toda systems on $\mathcal G$ an
 $\mathcal G^*$, respectively: 
\begin{equation}\label{eq:additiverationaltodadual}
 \sum_{k=1}^n \frac{\alpha_k- \alpha_{k+1}}{z_0 - z_{2k}}=0
 \quad\mbox{and}\quad  \sum_{k=1}^{n^{\prime}} \frac{\alpha_k - \alpha_{k+1}}{z^*_1 - z^*_{2k+1}} = 0.
\end{equation}
 Similar to the cross-ratio system, the additive rational Toda system can also be formulated by the matrix form.  However, one needs to introduce a 
\textit{gauge transformations} for $\Phi_-$ and $\Phi_-^*$. 
 This means that the wave function $\Phi_-$ 
  as in \eqref{eq:potentialU} and its dual wave function $\Phi_-^*$
 as in \eqref{eq:potentialUdual} should be gauged as
\begin{equation}\label{eq:gaugedPsi}
\left\{
\begin{array}{l}
 \Phi_-(v_i, \l)\longrightarrow \Mp_-(v_i, \l) = A(v_i, \l) \Phi_-(v_i, \l)
 \quad \mbox{for}\quad v_i \in \mathcal V(\mathcal G),\\[0.1cm]
 \Phi^*_-(v_i, \l)\longrightarrow \Mp^*_-(v_i, \l) = A^*(v_i, \l) 
 \Phi^*_-(v_i, \l)\quad \mbox{for} \quad v_i \in \mathcal V(\mathcal G^*),
\end{array}
\right.
\end{equation}
where 
\begin{equation}\label{eq:AAstar}
 A(v_i, \l) =
\begin{pmatrix}1 & H z_i \l^{-1} \\ 0    & 1 \end{pmatrix}, \quad 
A^*(v_i,\l) =\begin{pmatrix}1 &0 \\  - H z^*_i\l^{-1}    & 1 \end{pmatrix}.
\end{equation}
Then the transition matrices, i.e., the holomorphic potentials 
 $L_-(\mathfrak e, \l)$ and  $L_-^*(\mathfrak e, \l)$
 in \eqref{eq:Lminus} and \eqref{eq:Lminusdual}, respectively,   
 should be gauged accordingly:
\begin{equation}\label{eq:Ltilde}
\left\{
\begin{array}{l}
L_- (\mathfrak e, \l) \longrightarrow \mathcal L_-(\mathfrak e, \l)
 =A^*(v_j,\lambda) L_-(\mathfrak e, \l) (A(v_i,\l))^{-1}, \\[0.1cm]
L_-^*(\mathfrak e, \l)  \longrightarrow \mathcal L_-^*(\mathfrak e, \l) =A(v_j,\lambda) L_-^*(\mathfrak e, \l) (A^*(v_i,\l))^{-1}.
\end{array}
\right.
\end{equation}
 Here $\mathfrak e = (v_i, v_j) \in \mathcal E (\mathcal D)$.
The motivation of this gauge transformation becomes clear through 
 the following proposition which has been given 
 in \cite[]{Bobenko2002quad} except the third statement:
\begin{Proposition}\label{prp:keyprop}
\begin{figure}[ht]
\centering
\def\svgwidth{0.5\textwidth}
\begingroup%
  \makeatletter%
  \providecommand\color[2][]{%
    \errmessage{(Inkscape) Color is used for the text in Inkscape, but the package 'color.sty' is not loaded}%
    \renewcommand\color[2][]{}%
  }%
  \providecommand\transparent[1]{%
    \errmessage{(Inkscape) Transparency is used (non-zero) for the text in Inkscape, but the package 'transparent.sty' is not loaded}%
    \renewcommand\transparent[1]{}%
  }%
  \providecommand\rotatebox[2]{#2}%
  \newcommand*\fsize{\dimexpr\f@size pt\relax}%
  \newcommand*\lineheight[1]{\fontsize{\fsize}{#1\fsize}\selectfont}%
  \ifx\svgwidth\undefined%
    \setlength{\unitlength}{362.13967391bp}%
    \ifx\svgscale\undefined%
      \relax%
    \else%
      \setlength{\unitlength}{\unitlength * \real{\svgscale}}%
    \fi%
  \else%
    \setlength{\unitlength}{\svgwidth}%
  \fi%
  \global\let\svgwidth\undefined%
  \global\let\svgscale\undefined%
  \makeatother%
  \begin{picture}(1,0.6401674)%
    \lineheight{1}%
    \setlength\tabcolsep{0pt}%
    \put(0,0){\includegraphics[width=\unitlength,page=1]{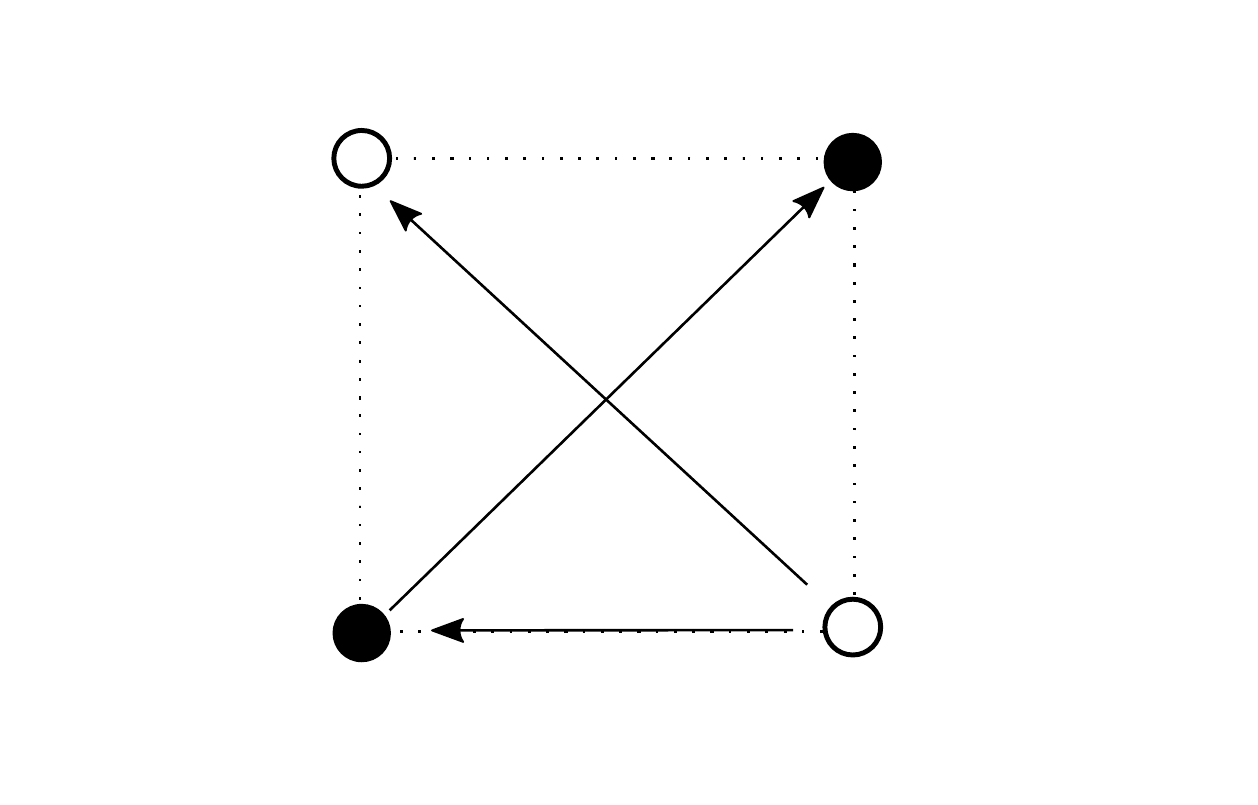}}%
    \put(0.69246864,0.07134023){\color[rgb]{0,0,0}\makebox(0,0)[lt]{\lineheight{1.25}\smash{\begin{tabular}[t]{l}$\Phi^*(v_1)$\end{tabular}}}}%
    \put(0.27619266,0.54405906){\color[rgb]{0,0,0}\makebox(0,0)[lt]{\lineheight{1.25}\smash{\begin{tabular}[t]{l}$\Phi^*(v_3)$\end{tabular}}}}%
    \put(0.6987868,0.53355699){\color[rgb]{0,0,0}\makebox(0,0)[lt]{\lineheight{1.25}\smash{\begin{tabular}[t]{l}$\Phi(v_2)$\end{tabular}}}}%
    \put(0.26359828,0.06502196){\color[rgb]{0,0,0}\makebox(0,0)[lt]{\lineheight{1.25}\smash{\begin{tabular}[t]{l}$\Phi(v_0)$\end{tabular}}}}%
    \put(0,0){\includegraphics[width=\unitlength,page=2]{transition.pdf}}%
  \end{picture}%
\endgroup%

\caption{Transition matrices for additive rational Toda system.}

\end{figure}
 In any quadrilateral $(v_0,v_1,v_2,v_3)$ on $\mathcal D$, 
 the following statements holds$:$
\begin{enumerate}
\item Let $\mathcal L_-^*(e_{\mathcal G^*},\l)$ be the diagonal transition 
 matrix from $\Mp_-^*(v_1, \l)$ to $\Mp_-^*(v_3, \l)$, i.e., 
\[
 \Mp_-^*(v_3, \l) = \mathcal L_-^*(e_{\mathcal G^*},\l) \Mp_-^*(v_1, \l).
\]
Then $\mathcal L_-^*(e_{\mathcal G^*},\l)$ only depends on $z_0= z(v_0)$, $z_2=z(v_2)$ 
and $\l$. Therefore the gauged wave function $\Mp_-^*$
 defined on $\mathcal V (\mathcal G^*)$ only depends on a solution of 
 the additive rational 
 Toda system on $\mathcal G$ and $\l$.

\item Let $\mathcal L_{-}(e_{\mathcal G},\l)$ be the 
 dual diagonal transition matrix 
 from $\Mp_-(v_0, \l)$ to $\Mp_-(v_2, \l)$, i.e,
\[
 \Mp_-(v_2, \l) = \mathcal L_-(e_{\mathcal G},\l) \Mp_-(v_0, \l).
\]
Then $\mathcal L_-(e_{\mathcal G},\l)$ only depends on $z_1^*= z^*(v_1)$, $z_3^*= z^*(v_3)$
 and $\l$. Therefore the gauged wave function $\Mp_-$
 defined on $\mathcal V (\mathcal G)$ only depends a solution of the dual additive rational 
 Toda system  on $\mathcal G^*$ and $\l$.
\item Let $\mathcal {L}_-(\mathfrak e,\l) $ be the transition matrix from
 $\Mp_-^*(v_1, \l )$ to $\Mp_-(v_0, \l)$, i.e., 
\[
 \Mp_-(v_0, \l) = \mathcal {L}_-(\mathfrak e,\l) \Mp_-^*(v_1, \l).
\]
 Then $\mathcal {L}_-(\mathfrak e,\l)$ only depends 
 on $z_0 = z(v_0)$, $z_1^* = z^*(v_1)$
 and $\l$.
\end{enumerate}
\end{Proposition}
\begin{proof}
 The statements (1) and (2) have been proven in \cite[]{Bobenko2002quad}, 
 but we give a brief proof for the sake of completeness.

(1): 
 Recall that the three-leg form in \eqref{eq:threeleg}
\[
 \frac{\alpha_1}{z_0 -z_1} -  \frac{\alpha_2}{z_0-z_3}  
 =  \frac{\alpha_1 - \alpha_2}{z_0 -z_2},
\]
 and its equivalent form 
\begin{equation}\label{eq:threeleg2}
\frac{\alpha_1}{z_0-z_1}z_1 -
\frac{\alpha_2}{z_0-z_3}  z_3
= \frac{\alpha_1-\alpha_2}{z_0-z_2}z_2.
\end{equation}
 From \eqref{eq:threeleg2}, it is easy to see 
 that 
\begin{equation}\label{eq:threeleg3}
\frac{\alpha_1}{z_0-z_1}z_0 -
\frac{\alpha_2}{z_0-z_3}  z_3
= \frac{\alpha_1 z_0-\alpha_2z_2}{z_0-z_2}\quad \mbox{and}\quad
\frac{\alpha_1}{z_0-z_1}z_1 -
\frac{\alpha_2}{z_0-z_3}  z_0
= \frac{\alpha_1 z_2-\alpha_2z_0}{z_0-z_2}
\end{equation}
hold. Then by using \eqref{eq:threeleg}, \eqref{eq:threeleg2} and 
 \eqref{eq:threeleg3}, we compute 
\begin{align}\label{eq:mathcalLminus*}
 \mathcal L_-^*(e_{\mathcal G^*},\l) &= 
\wt{L}_-(v_3,v_0,\l) \wt{L}_-(v_0,v_1,\l) \\ 
 & =\frac{1}{\mathfrak l }
 \left\{\id + \frac{\l^{-2}}{z_0 - z_{2}}
 \begin{pmatrix}
 H(\alpha_{2}z_{2} -\alpha_1 z_0) &H^2 (\alpha_1 - \alpha_{2})z_0 z_{2} \l^{-1} \\(\alpha_{2}- \alpha_1) \l
 &  H(\alpha_1 z_{2} -\alpha_{2} z_0)
 \end{pmatrix}\right\},
 \nonumber
\end{align}
where $\mathfrak l = \sqrt{(1 - H \alpha_1 \l^{-2})(1 - H \alpha_{2} 
 \l^{-2})}$. The claim follows.

(2): Similarly, we compute
\begin{align}\label{eq:mathcalLminus}
 \mathcal L_-(e_{\mathcal G},\l)  &= \wt{L}_-(v_0,v_1,\l )\wt{L}_-(v_1,v_2,\l)  \\
 & =
\frac{1}{\mathfrak l }
 \left\{\id + \frac{\l^{-2}}{z_{3}^* - z_{1}^*}
\begin{pmatrix}
 H(\alpha_{2}z_{1}^* -\alpha_1 z_{3}^*) & 
 (\alpha_{2}- \alpha_1) \l \\
H^2 (\alpha_1 - \alpha_{2})z_{1}^*z_{3}^* \l^{-1}
 &  H(\alpha_1 z_{1}^* -\alpha_{2} z_{3}^*)
 \end{pmatrix}\right\},
 \nonumber
\end{align}
 where $\mathfrak l = \sqrt{(1 - H \alpha_1 \l^{-2})(1 - H \alpha_{2} 
 \l^{-2})}$.

(3): For an edge $\mathfrak e = (v_1, v_0) \in \mathcal E (\mathcal D)$ 
 with $v_0 \in \mathcal G$ and $v_1 \in \mathcal G^*$, 
 we compute $\mathcal {L}_-(\mathfrak e, \l)$ as
\begin{align}\label{eq:initL}
&\mathcal {L}_- (\mathfrak e, \lambda)  = 
 \frac{1}{\sqrt{1 - H \alpha_1 \l^{-2}}}A^* (v_0, \l) L_-(v_0, v_1, \l) 
 (A(v_1, \l))^{-1} \\
\nonumber & =
\frac{1}{\sqrt{1 - H \alpha_1 \l^{-2}}}
 \begin{pmatrix}1 & 0 \\ - H z_0^* \l^{-1} & 1 \end{pmatrix} 
 \begin{pmatrix} 1 & H (z_1- z_0) \l^{-1}  \\ 
 (z_0^* -z_1^*) \l^{-1}& 1 \end{pmatrix}
 \begin{pmatrix} 1 & -H z_1 \l^{-1} \\ 0 & 1 \end{pmatrix}  \\
\nonumber & = 
\frac{1}{\sqrt{1 - H \alpha_1 \l^{-2}}}
\begin{pmatrix} 
 1  & - H z_0 \l^{-1}\\ -z_1^* \l^{-1}
& H (z_0 z_1^*- \alpha_1) \l^{-2} + 1
 \end{pmatrix}.
\end{align}
 Here we use the relation $z_1-z_0 = \frac{\alpha_1}{z_0^*-z_1^*}$, 
 which is equivalent with
 \[
  -z_1 z_1^* - z_0 z_0^* + z_0 z_1^* +z_1 z_0^* = \alpha_1.
 \]
 Thus $\mathcal {L}_- (\mathfrak e, \l)$ depends 
 only  $z_0$, $z_1^*$ and $\l$, 
 not $z_0^*$ and $z_1$.
\end{proof}
 
The above proposition implies that we only need a holomorphic data on $\mathcal{G}$, $z: \mathcal{V}(\mathcal{G})\rightarrow \C$ and they induce a dual holomorphic data on $\mathcal{G}^*$, i.e., $z^*:\mathcal{V}(\mathcal{G}^*) \rightarrow \C$. The function 
 $z^*$ is determined up to a global translation $z^*(v_i) \mapsto z^*(v_i) + c$ for some 
 constant $c \in \C$. We will see that this global translation will not affect 
 the CMC surfaces. 

From the above observation, we arrive the following definition.
\begin{Definition}\label{dfn:pairnormalized}
Let $\mathcal G$ (resp. $\mathcal G^*$) be a graph 
 with corner $\alpha$ which 
satisfies \eqref{eqn:toda1}, and $z_{2k}$ (resp. $z_{2k-1}^*$) be a solution of the (resp. the dual) additive rational Toda system
 in \eqref{eq:additiverationaltodadual}.
\begin{enumerate}
\item  The matrix-valued function 
 $\mathcal L_{-}$ in \eqref{eq:mathcalLminus} (resp. $\mathcal L_{-}^*$ in \eqref{eq:mathcalLminus*}) will be called the \textit{normalized 
 potential} on $\mathcal G$ (resp. $\mathcal G^*$), and the solution 
 $\Mp_-$ (resp. $\Mp_-^*$) is called the \textit{wave function} 
 (resp. \textit{dual wave function}) on $\mathcal G$ (resp. $\mathcal G^*$).
\label{itm:pairwave}
 
\item If the wave functions $\Mp_-$ and $\Mp_-^*$ in \eqref{itm:pairwave}
 are compatible, i.e., 
 they are connected by $\mathcal {L}_-(\mathfrak e, \l)$ as in 
 \eqref{eq:initL}, 
 then  
 the pair $(\Mp_-, \Mp_-^*)$ will be called  the \textit{pair of wave functions} on $(\mathcal G,\mathcal{G}^*)$.
Moreover, the pair $(\mathcal L_{-}, \mathcal L_{-}^*)$ will be called the \textit{pair of normalized potentials} on $(\mathcal G,\mathcal{G}^*)$.
\end{enumerate}
\end{Definition}

\section{Algebraic definition of Isothermic constant mean curvature surfaces for general graphs}\label{sc:Iso}

 In this section we give the \textit{algebraic definition} of 
 isothermic constant mean curvature surfaces for general graphs 
 utilizing the Weierstrass type representation. In particular we 
 first define the pair of extended frames of isothermic 
 constant mean curvature surfaces for general graphs 
 through the Iwasawa decomposition.
 We will then show that the pair of extended frames 
 define the pair of normalized potentials through the 
 Birkhoff decomposition.

\subsection{The Weierstrass type representaition for general graphs}

 We start from a solution 
 $z: \mathcal V (\mathcal G) \to 
 \C, z_i = z(v_i)$,  of the addtive rational Toda system 
 \eqref{eq:additiverationaltoda} on $\mathcal G$
 with a given corner function $\alpha=\alpha(\mathfrak e)$ 
 which satisfies 
\eqref{eqn:toda1}. 
 Moreover, let $\mathcal {L}_-$ be the 
 normalized potential and $\Mp_-$ the 
 the corresponding wave functions defined in 
 Definition \ref{dfn:pairnormalized} (1). 
\begin{Lemma}\label{thm:extendedgeneral} Retain notation the above. Perform the 
 Iwasawa decomposition of {\rm Theorem \ref{Thm:Iwasawa}} to 
$\Mp_-$ as $\Mp_-  = \Mp_{+} \Mp$
 i.e., 
 $\Mp\in \LSU$  and $\Mp_+\in \LSLP$.
 Then $\Mp$ satisfies 
\begin{gather}\label{eq:MP}
 \Mp(v_{2}, \l) =   \frac1{\tau(e_{\mathcal G}, \l)}\mathcal {U}(e_{\mathcal G}, \l) \Mp(v_{0}, \l),  
\intertext{where}
\label{eq:tildeUes}
 \mathcal {U} (e_{\mathcal G}, \l) = 
\begin{pmatrix}
 H a\l^{-2}   +  b+ H c\l^{2} 
 & - \bar r \l^{-1} - H \bar q \l -  
 H^2 \bar p\l^{3}
\\ 
 H^2 p\l^{-3}+ H q \l^{-1} +  r\l &
 H \bar c\l^{-2}   + \bar  b + H  \bar a\l^{2}
\end{pmatrix}
\end{gather}
 and $\tau (e_{\mathcal G}, \l) = 
 \sqrt{\mathcal U (e_{\mathcal G}, \lambda)} = \sqrt{(1 - H \alpha_1 \l^{-2}) (1 - H \bar \alpha_1\l^{2})
 (1 - H \alpha_2 \l^{-2})(1 -H \bar \alpha_2 \l^{2})}$.
 Here $a=a(e_{\mathcal G}),  b= b(e_{\mathcal G}), c= c(e_{\mathcal G}), 
p=p(e_{\mathcal G}), q= q(e_{\mathcal G})$ and $r= r(e_{\mathcal G})$
 depend only on fields of $\mathcal G^*$, i.e.,
 a solution of the dual additive rational Toda system on $\mathcal G^*$. 
\end{Lemma}

\begin{proof} 
From Proposition \ref{prp:keyprop}, it is clear that 
 $\Mp_{-} (v_{2k}, \l)$ only depend on the addtive rational Toda 
 system on $\mathcal G^*$.
 Therefore after Iwasawa decomposition for $\Mp_{-} (v_{2k}, \l)$
 the function $\Mp (v_{2k}, \l)$ only depends on a function 
 on $\mathcal G^*$.
 We now compute the Maurer-Cartan form $\Mp (v_{2k}, \l)$.
 On a quadrilateral $(v_0, v_1, v_2, v_3)$, a straightforward computation shows that 
\begin{align*}
 \mathcal {U}(e_{\mathcal G}, \l) &= 
 \tau(e_{\mathcal G}, \l) \Mp (v_2, \l)  
(\Mp (v_0, \l))^{-1} \\
&=
\tau (e_{\mathcal G}, \l) (\Mp (v_2, \l))^{-1}   
\mathcal L_{-} (e_{\mathcal G}, \l)
\Mp_+ (v_0, \l),
\end{align*}
 where $\mathcal L_{-}(e_{\mathcal G})$ 
 is given in \eqref{eq:mathcalLminus}. The right-hand side of 
 the above equation
 has the form 
\[
 \l^{-3} 
\begin{pmatrix}
0 & 0 \\
* & 0 
\end{pmatrix}
 + \l^{-2}
\begin{pmatrix}
* & 0 \\
0 & * 
\end{pmatrix}
 + \cdots.
\]
 Since the left-hand side of the above equation multiplying $\tau^{-1}$
 takes values in $\LSU$ and $\tau$ takes values in $\R_{>0}$, we have the 
 form as in the first equation of \eqref{eq:tildeUes}. 
\end{proof}
\begin{Remark}
 A similar statement holds for $\Mp_-^*$, i.e., for the Iwasawa decomposition 
 to 
\[
     \Mp_-^*  = \Mp_{+}^* \Mp^*, \quad 
 \Mp^*\in \LSU\quad\mbox{and}\quad \Mp_+^*\in \LSLP,
\] 
 and $\Mp^*$ satisfies
\begin{gather}
\label{eq:MPs} 
 \Mp^*(v_3, \l) =  \frac1{\tau(e_{\mathcal G^*}, \l)}
 \mathcal {U}^*(e_{\mathcal G^*}, \l) \Mp^*(v_1, \l), 
\intertext{where}
\nonumber
\mathcal U^* (e_{\mathcal G^*}, \l) = \begin{pmatrix}
 H a^*\l^{-2}   +  b^*+ H c^*\l^{2} 
 & - H^2 \overline{r^*} \l^{-3} - H \overline{q^*} \l^{-1} -  
 \overline{p^*}\l
\\ 
 p^*\l^{-1}+ H q^*  \l +  H^2 \overline{r^*}\l^{3} &
H \overline{c^*}\l^{-2}   + \overline{b^*} + 
 H \overline{a^*}\l^{2}
\end{pmatrix} 
\end{gather}
and
 $\tau (e_{\mathcal G^*}, \l) =\sqrt{\det \mathcal U^{*} (\mathfrak e_{\mathcal G^*}, \lambda)} = \sqrt{(1 - H \alpha_1 \l^{-2}) (1 - H \bar \alpha_1\l^{2})
 (1 - H \alpha_2 \l^{-2})(1 -H \bar \alpha_2 \l^{2})}$.
 Here $a^* = a^*(e_{\mathcal G^*}), b^*= b^*(e_{\mathcal G^*}), c^*= c^*(e_{\mathcal G^*}), 
 p^*= p^*(e_{\mathcal G^*}), q^*= q^*(e_{\mathcal G^*})$ and
 $r^*= r^*(e_{\mathcal G^*})$
 depend only on fields of $\mathcal G$, i.e.,  
 a solution of the dual additive rational Toda 
 equation on $\mathcal G$.
\end{Remark}

\begin{Proposition}\label{prp:pairext}
 Assume that normalized potentials $\mathcal L_-$ and $\mathcal L_-^*$ 
 are a pair of normalized potentials in the sense of Definition 
 {\rm \ref{dfn:pairnormalized} (2)}.
 Then the pair of maps $(\Mp,\Mp^*)$ given by the Iwasawa 
 decomposition to the pair of wave functions, i.e,  $(\Mp_-,\Mp_-^*) 
 = (\Mp_+,\Mp_+^*)(\Mp,\Mp^*)$, satisfies 
 the following system on any quadrilateral $(v_0, v_1, v_2, v_3)$ 
 on $\mathcal D = \mathcal G \cup \mathcal G^*:$
\begin{align}
\label{eqn:frame_general}\Mp^*(v_1^*,\lambda) = \frac1{\kappa(\mathfrak e, \lambda)}\mathcal{U}(\mathfrak e,\lambda) \Mp(v_0,\lambda), \quad 
\Mp(v_2,\lambda) = \frac1{\kappa(\mathfrak e, \lambda)}\mathcal{U}(\mathfrak e,\lambda)\Mp^*(v_1,\lambda),
\end{align}
where $\mathfrak e  =(v_0,  v_1)$ in the first equation and $\mathfrak e  =(v_1,  v_2)$ in the second equation and 
\begin{align*}
\mathcal{U}(\mathfrak e, \lambda) = \begin{pmatrix}
a_\mathfrak{e} + b_\mathfrak{e}\lambda^{-2} & c_\mathfrak{e} \lambda + d_\mathfrak{e}\lambda^{-1} \\ -\overline{c_\mathfrak{e}}\lambda^{-1} - \overline{d_\mathfrak{e}}\lambda  & \overline{a_\mathfrak{e}} + \overline{b_\mathfrak{e}}\lambda^2 \end{pmatrix}, \quad 
\mathcal{U}(\mathfrak e,\lambda) =\begin{pmatrix}
a_\mathfrak{e}^* + b_\mathfrak{e}^* \lambda^{2} & c_\mathfrak{e}^*  \lambda + d_\mathfrak{e}^* \lambda^{-1} \\ -\overline{c_\mathfrak{e}^* }\lambda^{-1} - \overline{d_\mathfrak{e}^* }\lambda  & \overline{a_\mathfrak{e}^* } + \overline{b_\mathfrak{e}^* }\lambda^{-2} \end{pmatrix}, 
\end{align*}
and 
\begin{align*}
\kappa(\mathfrak e, \lambda) =\sqrt{\det(\mathcal{U}(\mathfrak e,\lambda))} 
= \left\{ \begin{array} {l}
 \sqrt{(1-H\alpha_1\lambda^{-2})(1-H\bar \alpha_1\lambda^2)}, \quad 
(\mathfrak e= (v_0, v_1))  \\[0.2cm]
 \sqrt{(1-H\alpha_2\lambda^{-2})(1-H\bar \alpha_2\lambda^2)}, \quad (\mathfrak e= (v_1, v_2)) 
\end{array}
\right. .
\end{align*}
\end{Proposition}
\begin{proof}
 It is a straightforward computation similar to the proof of Lemma \ref{thm:extendedgeneral}.
\end{proof}
 Applying the Sym-Bobenko formula $\f$ in \eqref{eq:discretecmc} to $\Mp$ and 
 the dual $\f^{*}$ to $\Mp^*$, we naturally have maps $f$ and $f^*$ on 
 $\mathcal G$ and $\mathcal G^*$, i.e., we define $f = \f$ and $f^* = \f^*$, and call $f$ and $f^*$ the \textit{primal} surface 
 and the \textit{dual} surface, respectively.
 Finally we have the main theorem in this paper.
\begin{Theorem}[The generalized Weierstrass type representation on 
 general graphs]\label{thm:cmc_general}
 The pair of maps $(\Mp, \Mp^*)$ 
 in Proposition {\rm\ref{prp:pairext}} gives a pair of discrete isothermic 
 CMC surfaces $(f, f^*)$ through the Sym-Bobenko formula.
\end{Theorem}
 To prove Theorem \ref{thm:cmc_general}, we need to define 
 hyperedges. Similar to the proof of Lemma \ref{lem:cross_closing} we define 
 the hyperedges $E_{01}, E_{03}, E_{12}$ and $E_{32}$ as follows:
\begin{align}
\label{eqE01}
E_{01} &= -\frac{1}{\kappa_ 1H}(\Mp^*_1)^{-1}\cdot \mathfrak{U}_{01}\cdot \Mp_0,& 
E_{03} &=-\frac{1}{\kappa_ 1H} (\Mp^*_3)^{-1}\cdot \mathfrak{U}_{03}\cdot \Mp_0,&  \\
\label{eqE12}
E_{12} &=-\frac{1}{\kappa_ 1H} \Mp_2^{-1}\cdot \mathfrak{U}_{12}\cdot \Mp^*_1,& 
E_{32} &=-\frac{1}{\kappa_ 1H} \Mp_2^{-1}\cdot \mathfrak{U}_{32}\cdot \Mp_3^*,& 
\end{align}
where 
\begin{align*}
\mathfrak{U}_{01} 
&= \partial_t \mathcal U_{01} + \frac{\sqrt{-1}}{2}\sigma_3 \mathcal U_{01} 
 + \frac{\sqrt{-1}}{2} \mathcal U_{01} \sigma_3, 
\quad \mathfrak{U}_{03} 
= \partial_t \mathcal U_{03} -\frac{\sqrt{-1}}{2}\sigma_3 \mathcal U_{03} 
 - \frac{\sqrt{-1}}{2} \mathcal U_{03} \sigma_3, \\
\mathfrak{U}_{12} 
&= \partial_t \mathcal U_{12} - \frac{\sqrt{-1}}{2}\sigma_3 \mathcal U_{12} 
 - \frac{\sqrt{-1}}{2} \mathcal U_{12} \sigma_3, \quad
\mathfrak{U}_{32} 
= \partial_t \mathcal U_{32} + \frac{\sqrt{-1}}{2}\sigma_3 \mathcal U_{32} 
 + \frac{\sqrt{-1}}{2} \mathcal U_{32} \sigma_3.
\end{align*}
 In fact, it is easy to see that 
the imaginary part of $E_{ij}$ is $f_j - f^*_i$ or $f_j^* - f_i$
 from the construction, i.e., $E_{ij}$ are the hyperedges. 
 Then similar to Lemma \ref{lem:cross_closing}
 the following Lemma holds.
\begin{Lemma}
\label{lem:hyperedge}
Let $E_{01}$, $E_{12}$, $E_{23}$ and $E_{30}$ be the hyperedges, defined in 
\eqref{eqE01} and \eqref{eqE12}, that connect the 
 primal surface $f$ and the dual surface $f^*$. Then the 
following statement holds$:$
\begin{gather}
\label{eqn:hyperedge1}
|E_{32}| = |E_{01}| = \frac{\sqrt{\kappa_1^2 +\Re( \lambda^2 H \alpha_1)}}{\kappa_1 H},  \\ \label{eqn:hyperedge2}
\quad |E_{12}| = |E_{03}| = \frac{\sqrt{\kappa_2^2 +\Re( \lambda^2 H \alpha_2)}}{\kappa_2 H}, \\
\label{eqn:hyperedge3}
E_{12}\cdot E_{01} = E_{32} \cdot E_{03},
\end{gather}
 where $\kappa_1 = \kappa_{01} = \kappa_{32}$ and 
$\kappa_2 = \kappa_{03} = \kappa_{12}$.
Moreover, by using \eqref{eqn:hyperedge1},
 \eqref{eqn:hyperedge3} is equivalent with 
\begin{equation}\label{eqn:hyperedge4}
E_{02}\cdot E_{23} = E_{10} \cdot E_{02}.
\end{equation}
 Note that $E_{ji} = \overline{E_{ij}}$ 
 with $(ij) \in \{(01),(12),(32),(03)\}$. 
 Similary 
\begin{equation}\label{eqn:hyperedge5}
E_{32}\cdot E_{03} = E_{12} \cdot E_{01} 
\end{equation}
holds, and by using \eqref{eqn:hyperedge2}, 
 \eqref{eqn:hyperedge5} is equivalent with
\begin{equation}\label{eqn:hyperedge6}
 E_{02}\cdot E_{12} = E_{03} \cdot E_{02}.
\end{equation}
\end{Lemma}
\begin{proof}
 The proof is verbatim to the proof of Lemma \ref{lem:cross_closing}, thus 
 we omit.
\end{proof}

\begin{proof}[Proof of Theorem {\rm \ref{thm:cmc_general}}]
 It is easy to see that \eqref{eqn:hyperedge4} and \eqref{eqn:hyperedge6} 
 imply 
\[
f^*_{3} - f_2 = E^{-1}_{02} \cdot (f_0-f^*_1) \cdot E_{02} \quad \mbox{and} \quad
f^*_{3} - f_0 = E^{-1}_{02} \cdot (f_2-f^*_1) \cdot E_{02},
\]
respectively. Thus Definition \ref{def:constantmean} (1) holds.
Let us verify the isothermic condition \eqref{eq:dis_iso} in Definition \ref{def:genIso}:
\begin{align*}
E_{02} \cdot E^*_{31} &= E_{02} \cdot (E_{30} + E_{01}) \\
&= E_{02} \cdot E_{30} + E_{32} \cdot E_{02} \\
&= (E_{32} - E_{30}) \cdot E_{30} +  E_{32} \cdot (E_{32} - E_{30})\\
&= -E_{30}\cdot E_{30} + E_{32} \cdot E_{32}.
\end{align*}
 In the second equality, we use the expression in \eqref{eqn:hyperedge4}. 
By symmetry, two terms above will be canceled with the terms in neighboring 
 parallelogram. 
 Thus \eqref{eq:dis_iso} holds. This completes the proof
\end{proof}
 For the pair of maps $(\Mp, \Mp^*)$ as above,
 we are now going to show that even though the gauge transformation in 
 \eqref{eq:Ltilde} changes the pair of normalized 
 potentials and thus changes the extended frame, 
 it does not affect the surface $f$ on $\mathcal{G}$ 
 and the surface $f^*$ on $\mathcal G^*$.
\begin{Lemma}\label{lem:Gaugepair}
 Let $(\Mp, \Mp^*)$ be a pair of matrices determined from 
 a pair of normalized potentials $(\mathcal L_-, \mathcal L_-^*)$.
 Then there exists an extended frame $\Phi$ 
 as in {\rm Theorem  \ref{thm:DPW}} and 
 a matrix $G: \mathcal V (\mathcal D) \to \LSU$ such that 
 the followings hold$:$
\begin{align}\label{eq:Phiandtilde}
\left\{
\begin{array}{l}
 \Mp^*(v_{i}, \l) = G^*(v_{i}, \l) \Phi(v_{i}, \l) \quad \mbox{for} 
 \quad v_i \in \mathcal V(\mathcal G^*) \\[0.1cm]
 \Mp (v_{i}, \l) = G(v_{i}, \l) \Phi(v_{i}, \l) 
\quad \mbox{for} 
 \quad v_i \in \mathcal V(\mathcal G)
\end{array}
\right..
\end{align}
 Moreover,  $G^*$ and $G$ have the following forms$:$
\begin{equation}\label{eq:gaugeG}
\left\{
\begin{array}{l}
 G^*(v_{i}, \l) = \dfrac{1}{\sqrt{|p_i|^2 + H^2 |q_i|^2}}
 \begin{pmatrix} p_i &H  q_i \l^{-1}\\- H \bar q_i \l & \bar p_i \end{pmatrix}
 \quad \mbox{for $v_i \in \mathcal V(\mathcal G^*)$}
\\[0.1cm]
 G(v_{i}, \l)=\dfrac{1}{\sqrt{|p_i|^2 + H^2 |q_i|^2}}
 \begin{pmatrix} p_i & H q_i \l\\- H \bar q_i \l^{-1} & \bar p_i \end{pmatrix}
 \quad \mbox{for $v_i \in \mathcal V(\mathcal G)$}
\end{array}
\right.,
\end{equation}
 where $p_i= p(v_i)$ and $q_i=q(v_i)$ are some complex valued functions.
\end{Lemma}
\begin{proof}
 Recall that $\Mp_-^*$ and  $\Mp_-$ are the gauged wave functions 
 as in \eqref{eq:gaugedPsi}:
\[
\left\{
\begin{array}{l}
 \Mp_-^*(v_i, \l) = A^*(v_i, \l)\Phi_-^*(v_i, \l) \quad \mbox{for} \quad 
 v_i \in \mathcal V(\mathcal G^*), \\[0.1cm]
  \Mp_-(v_i, \l) = A (v_i, \l)\Phi_-(v_i, \l)
 \quad \mbox{for}\quad 
 v_i \in \mathcal V(\mathcal G).
\end{array}
\right.
\]
 We now compute the Iwasawa decompositions for 
 $\Mp_-^*$ and   $\Mp_-$ as follows:
 First we decompose $\Phi_-^*$ and   $\Phi_-$ as 
\[
\Phi_-^*= \Phi_{+}^* \Phi^*\quad \mbox{and} \quad 
 \Phi_-  = \tilde \Phi_{+} \tilde \Phi,
\]
 where $\Phi^*, \tilde \Phi \in \LSU$ and $\Phi_+^*, \tilde 
 \Phi_+ \in \LSLP$. 
 Note that $\Phi_-^*$ is defined on $\mathcal V(\mathcal G^*)$ and
 $\tilde \Phi_-$ is defined  on $\mathcal V(\mathcal G)$, 
 thus $\Phi^*$ and $\tilde \Phi$ together give the 
 extended frame on $\mathcal V (\mathcal D) = \mathcal V(\mathcal G) \cup  \mathcal V(\mathcal G^*)$ as in Theorem \ref{thm:DPW}. i.e., 
 we define $\Phi$ as $\tilde \Phi$ on $\mathcal V(\mathcal G)$ 
 and $\Phi^*$ on $\mathcal V(\mathcal G^*)$. 
 Then we have 
\[
\Mp_-^*= A^* \Phi_{+}^* \Phi^*\quad \mbox{and} \quad 
 \Mp_-  = A \tilde \Phi_{+} \tilde \Phi.
\]
 We now decompose $A^* \Phi_{+}^*$ and  $A \tilde \Phi_{+}$ as
\[
A^* \Phi_{+}^* =\hat \Phi_{+}^* \hat \Phi^*\quad \mbox{and} \quad 
A \tilde \Phi_{+} = \hat \Phi_{+} \hat \Phi,
\]
 where $\hat \Phi^*, \hat \Phi \in \LSU$ and $\hat \Phi_+^*, \hat 
 \Phi_+ \in \LSLP$. 
 Since $A^*$ and $A$ have the special forms as in \eqref{eq:AAstar}, 
 the unitary parts $\hat \Phi$ and $\hat \Phi^*$ are given in 
 \eqref{eq:gaugeG} i.e., we define the functions as 
 $G^* = \hat \Phi^*$ for $v_i \in \mathcal V(\mathcal G^*)$ and 
 $G= \hat \Phi$ for 
 $v_i \in \mathcal V(\mathcal G)$.
\end{proof}
\begin{Theorem}[Sub-lattice theorem]
 Retain the assumptions in Lemma {\rm \ref{lem:Gaugepair}}, 
 and take $f$ and $f^*$ the primal and the dual 
 isothermic CMC surfaces in Theorem {\rm \ref{thm:cmc_general}}.
 Moreover, define surfaces $\tilde \f$ and $\tilde \f^*$ by the 
 Sym-Bobenko formula in \eqref{eq:discretecmc} applied to the extended 
 frame $\Phi$ in Lemma {\rm \ref{lem:Gaugepair}}.
 Then the following equalities hold$:$
 \[
  f = \tilde \f|_{\mathcal V(\mathcal G)} \quad \mbox{and}\quad
  f^* = \tilde \f^*|_{\mathcal V(\mathcal G^*)}.
 \]
\end{Theorem}
\begin{proof}
 Plug the relations \eqref{eq:Phiandtilde} into 
 the Sym-Bobenko formulas $\f$ in the first formula in 
 \eqref{eq:discretecmc} and
 $\f^*$ in the second formula in \eqref{eq:discretecmc}, respectively. 
 A direct computation shows that the term $G^*$ and $G$  do not 
 affect on the resulting surfaces  $\f$ and $\f^*$, respectively, 
 and the claims follows.
\end{proof}
\subsection{Algebraic definition of constant mean curvature surfaces
 on general graphs}\label{sbsc:extendgeneral}
 From the previous section, we arrive the following 
 definition.
\begin{Definition}[Algebraic definition of constant mean curvature
 surfaces on general graphs]
 Let  $(\Mp, \Mp^*)$ be a pair of maps defined in 
 \eqref{eq:MP} and \eqref{eq:MPs} and assume that 
 it is compatible, i.e., it satisfies the condition
 \eqref{eqn:frame_general}. Then $(\Mp, \Mp^*)$ will
 be called the \textit{pair of extended frame} and 
 the resulting surfaces through the Sym-Bobenko formulas
 (for the case of a CMC surface)  or a direct calculation 
 (for the case of a minimal surface)
 will be called the isothermic \textit{constant mean curvature 
 surfaces} on $(\mathcal G, \mathcal G^*)$.
\end{Definition}
\begin{Remark}
 From the proof of Theorem \ref{thm:cmc_general}, it is clear that 
 discrete constant mean curvature surfaces in the above algebraic definition 
 satisfies the geometric definition in Definition \ref{def:constantmean}. However, 
 it has not been known that discrete constant mean curvature surfaces satisfy 
 the algebraic definition or not.
\end{Remark}
 We will finally show that the pair of extended frames give the pair of
 normalized potentials. 
\begin{Theorem}
 Let $(\Mp, \Mp^*)$ be a pair of extended frames.
 Perform the Birkhoff decomposition 
 of {\rm Theorem \ref{Thm:Birkhoff}} to $\Mp$ and $\Mp^*:$
\begin{equation}\label{eq:BirkhoffofFtilde}
\Mp^* = \Mp_+^* \Mp_-^*, \quad 
\Mp = \Mp_+ \Mp_-, \quad 
\end{equation}
where $\Mp_- :\mathcal V (\mathcal G) \to \LSLNN$, 
$\Mp_-^* :\mathcal V (\mathcal G^*) \to \LSLNN$, 
 $\Mp_+ :\mathcal V (\mathcal G) \to  \LSLP$ and 
$\Mp_+^* :\mathcal V (\mathcal G^*) \to  \LSLP$.
 Then $(\Mp_-, \Mp_-^*)$ is a  pair of normalized potentials.
\end{Theorem}

\begin{proof}
 Since $\Mp$ and $\Mp^*$ are compatible, 
 at each point $v_{2k+1} \in \mathcal V (\mathcal G^*)$
 or $v_{2k} \in \mathcal V (\mathcal G)$, we find gauge matrices 
 and a $\Phi$ satifies the relation as in \eqref{eq:gaugeG} 
 such that 
\begin{equation}\label{eq:Phirelation}
  \Phi(v_{2k+1}) \Phi(v_{2k})^{-1} = 
\frac{1}{\sqrt{(1- H \alpha_k \l^{-2})(1- H \bar \alpha_k \l^2)}}
\begin{pmatrix}
* &  *\l^{-1} + *\l\\
*\l^{-1} + *\l &  * 
\end{pmatrix},
\end{equation}
 where $*$ denote the function of $\mathfrak e$ independent of $\l$.
 Thus $\Phi$ is in fact the form of the extended frame in the sense of 
 Definition \ref{def:extended}.

 We now decompose $\Mp$ and $\Mp^*$ in two steps: First decompose $\Phi = \Phi_+ \Phi_-$ by the Birkhoff decomposition, we have 
\[
 \Mp = G \Phi_+ \Phi_- \quad \mbox{and} \quad 
 \Mp^* = G^* \Phi_+ \Phi_-,
\]
 respectively.
 Second we decompose $G \Phi_+ $ and $G \Phi_+ $ as
 $G \Phi_+ =\hat \Phi_+ \hat \Phi_-$ and 
 $G^* \Phi_+ =\hat \Phi_+^* \hat \Phi_-^*$, thus
\begin{equation}\label{eq:decompMp}
 \Mp = \hat \Phi_+ (\hat \Phi_- \Phi_-), \quad 
 \Mp^* = \hat \Phi_+^* (\hat \Phi_-^* \Phi_-),
\end{equation}
 i.e., $\Mp_- = \hat \Phi_- \Phi_-$ and $\Mp^*_- 
 = \hat \Phi_-^* \Phi_-$.
 From the forms of $G$ and $G^*$ we know that $\hat \Phi_-$
 and $\hat \Phi_-^*$ are an upper triangular matrix and 
 a lower triangular matrix:
\[
 \hat \Phi_-(v_i, \l)= \begin{pmatrix} 1 &  a(v_i) \l^{-1} \\ 0 & 1 \end{pmatrix}, \quad 
 \hat \Phi_-^*(v_i, \l)= \begin{pmatrix} 1 & 0 \\ b(v_i) \l^{-1} & 1\end{pmatrix}.
\]
 Moreover, we compute the Maurer-Cartan form of $\Phi_-$
 by using the form of $\Phi$ (see the proof of Theorem 
 \ref{thm:normalized} in Appendix \ref{subsec:birk_quad}):
\begin{align*}
 L_{-}(\mathfrak e, \l) &= 
\Phi_-(v_j, \l) (\Phi_-(v_i, \l) )^{-1} \\
&   = \dfrac{1}{\sqrt{1 - H \alpha(\mathfrak e) \l^{-2}}}
\begin{pmatrix} 
 1 &  H (z_i-z_{j})\l^{-1} \\
\dfrac{\alpha(\mathfrak e)} {z_i-z_j} \l^{-1}  &  1
\end{pmatrix}.
\end{align*}
 Note that $z_i$ is a solution of the cross-ratio system 
 on $\mathcal V(\mathcal D)$.
 We now consider the Maurer-Cartan forms of $\Mp_-$ and 
 $\Mp_-^*$ as
 \[
  \mathcal L_{-}(e_{\mathcal G}, \l) = 
\Mp_-(v_j, \l) (\Mp_-(v_i, \l) )^{-1} \quad 
 \mbox{and } \quad
 \mathcal L_{-}^*(e_{\mathcal G^*}, \l) = 
\Mp_-^*(v_j, \l) (\Mp_-^*(v_i, \l) )^{-1},
 \]
 respectively.
 Since $\Mp$ and $\Mp^*$ do not depend on $\mathcal V(\mathcal G)$ 
 and $\mathcal V(\mathcal G^*)$, 
 thus $\mathcal L_{-}(e_{\mathcal G}, \l)$ and 
 $\mathcal L_{-}^*(e_{\mathcal G^*}, \l)$ do not depend 
 $\mathcal V(\mathcal G)$ and $\mathcal V(\mathcal G^*)$.
 Then it is easy to see that on any quadrilateral $(v_0, v_1, v_2, v_3)$, 
 we have 
\begin{align*}
\mathcal L_{-}(e_{\mathcal G}, \l) = 
\
\hat \Phi_-(v_2, \l) L_{-}(v_1, v_2, \l)  L_{-}(v_0, v_1, \l) \hat \Phi_-(v_0, \l)^{-1}.
\end{align*}
 Since the upper triangular gauge $\hat \Phi_-$ is determined 
 so that $\mathcal L_{-}$ does not depend on $\mathcal G$, we 
 have $a(v_i) = H z_i$. Similarly, we have $b(v_i) = - H z_i$.
 Therefore, $\mathcal L_{-}(e_{\mathcal G}, \l)$ and 
 $\mathcal L_{-}^*(e_{\mathcal G^*}, \l)$ are 
 a normalized potential and a dual normalized potential, respectively.
 This completes the proof.
\end{proof}

\section{Summary and open questions}
\begin{figure}[ht]
\def\svgwidth{0.8\textwidth}
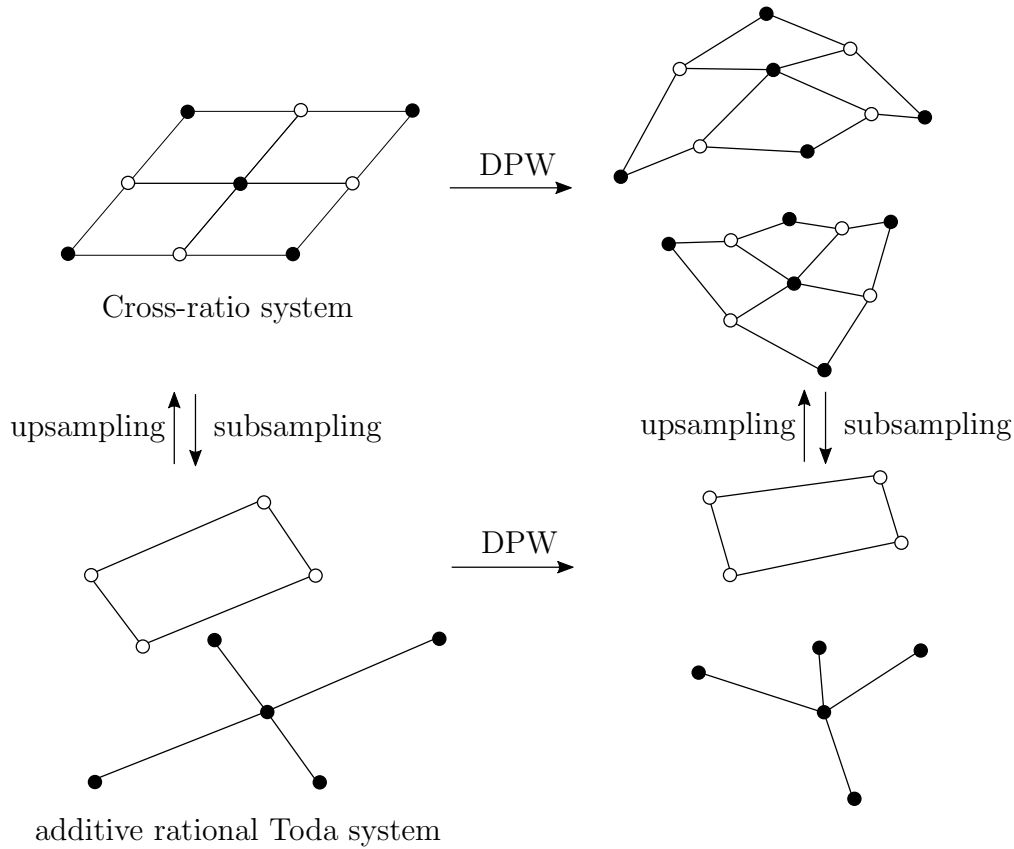
\caption{The relation between the surfaces over quadrilateral graphs and general graphs}\label{fig:summary}
\end{figure}

We define the discrete isothermic surface over general graphs and subsequently the corresponding minimal surface and non-zero constant mean curvature surface, which generalize and unify the existing discrete surfaces in \cite{bobenko1999discretization, bobenko1996isothermic, Hoffmann2018face, Lam2016isothermic, Lam2016minimal, Hoffmann2016constraint}. Moreover, we show that our constant mean curvature surfaces can be obtained by a discrete DPW method applied on the additive rational Toda system, which can be understood as discrete holomorphic data over general graphs. In particular, the relation to 
the DPW method on a quad graph $\mathcal D$ in \cite{Hoffmann1999cmc} is illustrated in Figure~ \ref{fig:summary}.

One open question is how to find a variational characterization of the discrete CMC surface and the connection between the CMC surfaces from variational principle and integrable system. Currently, such a connection for minimal surfaces has been constructed by Lam \cite{Lam2016minimal}, but it is unclear how to extend it to non-zero CMC surfaces. The variational principle for discrete CMC surfaces has been studied by Polthier and Rossman from a numerical perspective \cite{Polthier_2002}. Yet the connection to integrable system is still missing. Inspired by the recent work by Bobenko and Romon on discrete Lawson correspondence \cite{Lawson2017}, which  shows the equivalence between discrete CMC surfaces in $\mathbb{R}^3$ and discrete minimal surfaces in $S^3$, we expect that the variational properties will arise in $S^3$.

\appendix
\section{Loop groups}\label{app:loopgroups}
In this section we collect basic definitions  
 of loop groups and their decompositions, the so-called 
 Birkhoff and Iwasawa decompositions, Theorems \ref{Thm:Birkhoff}
 and \ref{Thm:Iwasawa}, respectively.

 Let $\SL$ be the complex special linear Lie group of degree two 
 and the special unitary group of degree two $\SU$ as a real form of $\SL$. 
 Then the twisted loop group of $\SL$ is a space of smooth maps 
\[
 \LSL = \{\gamma : S^1 \to \SL \mid \mbox{$\gamma$ is smooth and 
 $\sigma (\gamma(-\l)) = \gamma (\l)$}\}.
\]
 Here $\sigma$ is an involution on $\SL$.  In this paper, $\sigma$ is
 explicitly given by 
\[
 \sigma (g)  = \ad (\sigma_3) g, \quad g \in \SL.
\]
 We can introduce a suitable topology such that $\LSL$ is a Banach Lie group, 
 see  \cite{pressley1986loop}. We then define two subgroups of $\LSL$ as follows:
 \begin{align*}
 \LSLP &= \left\{\gamma \in \LSL \mid \mbox{$\gamma$ can be extended holomorphically to 
 the unit disk}\right\}, \\
 \LSLN &= \left\{\gamma \in \LSL \mid \begin{array}{l}
 \mbox{$\gamma$ can be extended holomorphically to}  \\
\mbox{ outside of the unit disk and $\infty$}
\end{array}
\right\}.
\end{align*}
 Moreover, denote  $\LSLPN$ and $\LSLNN$ respectively the subgroups of 
 $\LSLP$ and $\LSLN$ with normalization $\gamma(0)= \id$ (for $\gamma \in \LSLP$) and 
 $\gamma(\infty)= \id$ (for $\gamma \in \LSLN$).
 Finally the twisted loop group of $\SU$  is 
\[
 \LSU = \{\gamma : S^1 \to \SU \mid \mbox{$\gamma$ is smooth and 
 $\sigma (\gamma(-\l)) = \gamma (\l)$}\}.
\]
 The following two decomposition theorems are fundamental for the loop groups:
\begin{Theorem}[Birkhoff decomposition \cite{pressley1986loop}]
\label{Thm:Birkhoff}
The following multiplication maps are respectively 
 diffeomorphisms onto its images$:$ 
\begin{equation*}
\LSLPN \times \LSLN \to \LSL
\quad\text{and}\quad
\LSLNN  \times \LSLP \to \LSL.
\end{equation*}
Moreover, the images $\LSLPN \cdot \LSLN $
and $\LSLNN  \cdot \LSLP$ are both open and dense in $\LSL$, and 
 they are called the \emph{big cells}.
 Therefore, for any element $g \in \LSL$ in the big cell, 
 there exist $g_{+} \in \LSLPN, g_- \in \LSLN$, 
 $h_{-} \in \LSLNN$ and $h_+ \in \LSLP$ such that 
\begin{equation}\label{Bikrhoff}
 g = g_{+} g_{-} = h_{-} h_{+}
\end{equation}
 hold.
\end{Theorem}

\begin{Theorem}[Iwasawa decomposition \cite{pressley1986loop}]
\label{Thm:Iwasawa}
 The following multiplication map is a diffeomorphism onto 
 $\LSL:$
\begin{equation*}
\LSLP \times \LSU  \to \LSL.
\end{equation*}
 Therefore, for any element $g \in \LSL$, there exist
 $g_u \in \LSU$ and $g_{+} \in \LSLP$ such that 
\begin{equation}\label{eq:Iwasawa}
 g = g_{+}g_u
\end{equation}
 holds.
\end{Theorem}

\section{Compatibility conditions on quad graphs}
 In this section we discuss the compatibility condition 
 of \eqref{eq:dextended} on 
 an elementary quadrilateral $(v_0, v_1, v_2, v_3)$, i.e.,
 it is 
\[
 U(v_{1}, v_{2}, \alpha_{2}, \l) U(v_{0}, v_{1}, \alpha_{1}, \l)
 =  U(v_{3}, v_{2}, \alpha_1, \l) U(v_0, v_3, \alpha_2, \l),
\]
 and it can be computed as 
\begin{gather}
\label{eq:comp1}   H^2 (
\bar u_{03} u_{32}- \bar u_{01}  u_{12} )+
 d_{12} d_{01} -d_{32} d_{03}
 +\frac{\bar \alpha_{1} \alpha_2}{u_{03} \bar u_{32}}  
 -\frac{\alpha_{1} \bar \alpha_{2}}{u_{01} \bar u_{12}} =0, \\
\label{eq:comp2} H\left(\frac{\alpha_1}{\alpha_2}- 
 \frac{u_{01} u_{32}}{u_{12}u_{03}}
\right)  = 0, \\
\label{eq:comp3} H (- d_{12} u_{01} - \bar d_{01} u_{12} 
 + d_{32} u_{03} +\bar d_{03} u_{32})   =0,  \\
\label{eq:comp4}  
\frac{\alpha_1}{\alpha_2} 
 - \left(\frac{u_{01}u_{32}}{ u_{12}u_{03}}\right)
 \frac{ d_{01} u_{03} -\bar d_{32} u_{12}}{
 d_{03} u_{01} - \bar d_{12} u_{32}}
 =0.
\end{gather}
 Here $d_{ij}$ and $u_{ij}$ denote the functions of 
 the edge $\mathfrak e = (v_i, v_j)$.
 As we have discussed in Remark \ref{rm:kappa}, we do not need to 
 assume the form of $\kappa$ as in \eqref{eq:kappa}, i.e., 
 \eqref{eq:comp1},  \eqref{eq:comp2},  \eqref{eq:comp3} and \eqref{eq:comp4}
 are exactly the necessary sufficient conditions of existence of 
 the extended frame $\Phi$.

 \textbf{The case $H =0$:} 
 The compatibility conditions are simplified as 
\begin{gather}
 \label{eq:comp1min}   
 d_{12} d_{01} -d_{32} d_{03}
 +\frac{\bar \alpha_{1}  \alpha_2}{u_{03} \bar u_{32}}  
 -\frac{\alpha_{1} \bar \alpha_{2}}{u_{01} \bar u_{12}} =0, \\
\label{eq:comp4min}  
\frac{\alpha_1}{\alpha_2} 
 - \left(\frac{u_{01}u_{32}}{ u_{12}u_{03}}\right)
 \frac{d_{01} u_{03} -\bar d_{32} u_{12}}{
 d_{03} u_{01} - \bar d_{12} u_{32}}
 =0.
\end{gather}
 Let $z_i : \mathcal V (\mathcal D) \to \C$ be a solution of the cross-ratio 
 system in \eqref{eq:def-cross} and $z_i^*: \mathcal V (\mathcal D) \to \C$ be 
 the dual solution:
\[
 z_i^* - z_j^* =  \frac{\alpha_{ij}}{z_j - z_i}.
\]
 Then setting $u$ and $d$ as 
\begin{align*}
 u_{ij} = \frac{\alpha_{ij} }{z_i^* - z_j^*}\sqrt{(1 + |z_i^*|^2)(1 + |z_j^*|^2)} 
 \quad 
 \mbox{and}\quad
 d_{ij} = \frac{1+\bar z_i^* z_j^*}{\sqrt{(1 + |z_i^*|^2)(1 + |z_j^*|^2)}}.
\end{align*}
 Then \eqref{eq:comp1min} and \eqref{eq:comp4min} are clearly satisfied.
 Note that $\alpha_{01} = \alpha_{32}= \alpha_1$ and 
 $\alpha_{03} = \alpha_{12}= \alpha_2$. In fact the extended frame $\Phi$ 
 can be  explicitly given in \eqref{eq:Phizero}.

 \textbf{The case $H \neq 0$:} Without loss of generality, 
 we can assume $H =1$.
 Introduce a function $w$ on the vertices and set 
\[
 u_{ij} = \sqrt{\alpha_{ij}} w_i w_j.
\]
 Then \eqref{eq:comp2} is 
 clearly satisfied and \eqref{eq:comp3} and \eqref{eq:comp4} can 
 be simplified as 
\begin{gather*}
 - d_{12}  u_{01} - \bar d_{01} u_{12} 
 + d_{32} u_{03} +\bar d_{03} u_{32}   =0,  \\
\bar d_{03} \bar u_{01} - d_{12} \bar u_{32}
 -\bar d_{01} \bar u_{03} +d_{32} \bar u_{12} =0.
\end{gather*}
 Then one can solve the above equations for $d_{12}$ and $d_{32}$.
 Plugging the solutions into \eqref{eq:comp1}, we have 
 equation for $w_{2}$, and it can be solved.

\section{Isothermic parametrized constant mean curvature surfaces on 
 quadrilateral graphs}\label{subsc:iso}
 In this section we recall some basic definitions of isothermic parametrized 
 minimal surfaces and CMC surfaces. 
\begin{Definition}[Definition 17 in \cite{bobenko1999discretization}]\label{def:classicaliso}
Let $\f_0$, $\f_1$, $\f_2$, $\f_3$ be four points in 
 $\R^3\subset \Im(\mathbb{H})$. The spin cross-ratio is defined by
\begin{align*}
Q(\f_0, \f_1, \f_2, \f_3) :=  (\f_0 -\f_1) (\f_1 -\f_2)^{-1}(\f_2 -\f_3) (\f_3 -\f_0)^{-1}
\end{align*}
 Then a discrete surface $\f: \mathcal V (\mathcal D) \rightarrow \R^3$ is a 
 \textit{discrete 
 isothermic parametrized surface} if all elementary quadrilaterals 
 have factorized spin cross-ratios:
\begin{align*}
 Q(\f_0, \f_1, \f_2, \f_3) := -\frac{\kappa_1^2}{\kappa_2^2}<0 ,
\end{align*}
 where $\kappa$ is a labelling on $\mathcal D$.
\end{Definition}
 In particular, $ Q(\f_0, \f_1, \f_2, \f_3)=-1$
 give a special type of isothermic parametrized surface 
 \cite[Definition 18]{bobenko1999discretization}.
 We now give definitions of isothermic parametrized minimal surfaces
 and isothermic parametrized CMC surfaces on $\mathcal D$.

\begin{Definition}[Definition 20 in \cite{bobenko1999discretization}]\label{df:classmin}
 A discrete isothermic parametrized surface $\f:  \mathcal V (\mathcal D) \to \R^3$ is 
 called 
a {\it discrete isothermic parametrized minimal surface} if there is a dual 
 discrete isothermic parametrized surface $\mathfrak n= \f^* :  \mathcal V (\mathcal D) \to S^2$ and some function  $\Delta$ on $\mathcal V(\mathcal D)$ 
 such that
\begin{align}\label{eq:}
 \langle d \f(\mathfrak e), \mathfrak n (v_0)\rangle  = \pm \frac{\Delta (v_0)}{\kappa(\mathfrak e)^2}, \quad 
 \langle d \f(\mathfrak e^{\prime}), \mathfrak n (v_0)\rangle =\mp  \frac{\Delta (v_0)}{\kappa(\mathfrak e^{\prime})^2}.
\end{align}
 holds for any adjacent edges $\mathfrak e$ and $\mathfrak e^{\prime}$
 starting from the vertex  $v_0 \in \mathcal V (\mathcal D)$. Since 
 $\mathcal D$ is an even quadrilateral graph, the sign is well-defined.

\end{Definition}
\begin{Definition}[Definition 11 in \cite{bobenko1999discretization}]\label{df:classCMC}
 A discrete isothermic parametrized surface $\f: \mathcal V (\mathcal D) \to \R^3$ is 
 called 
a {\it discrete isothermic parametrized CMC surface} if there is a dual 
 discrete isothermic parametrized surface $\f^* :  \mathcal V (\mathcal D) \to \R^3$ and a non-zero constant $H$ such that 
 \begin{equation}\label{eq:defCMC}
\| \f - \f^*\|^2 = \frac{1}{H^2}
\end{equation}
 holds.
\end{Definition}
 We now give the proof of {\rm Theorem \ref{thm:extendedandCMC}}:
 First note that we abbreviate $f(v_i)$ by $f_i$, 
 $a(\mathfrak e)$ by $a_{ij}$ 
 and so on.

$(1)$ Since the Hopf differential $\alpha$ is labelled
 on the quadrilateral $(v_0, v_1, v_2, v_3)$, we have 
 $\alpha_1 = \alpha_{01} =\alpha_{32}$ 
 and $\alpha_2 =\alpha_{12} =\alpha_{03}$, see Figure \ref{fig:cr}. 
 A straightforward computation shows that 
\[
 Q(\f_0, \f_1, \f_2, \f_3)= 
 -\frac{\kappa_1^2}{\kappa_2^2},
\]
 where $\kappa_1^2 = \kappa_{01}^2 = \kappa_{32}^2 = (1 - 
 H \alpha_1)^2$ and $\kappa_2^2 = \kappa_{03}^2 = \kappa_{12}^2 
= (1 -  H \alpha_2)^2$.
 Therefore $\f$ is isothermic parametrized. We now define $\f^*$
 as 
\begin{equation}\label{eq:discretedualcmc}
 \f^* = -\frac{1}{H}\left( \Phi^{-1} \partial_{t} \Phi   - \frac{\sqrt{-1}}{2}\ad \Phi^{-1} (\sigma_3) \right)\Big|_{t=0}, \quad \l =e^{it}.
\end{equation}
 Then an another straightforward computation shows that 
\[
 Q(\f^*_0, \f^*_1, \f^*_2, \f^*_3)=- 
 \frac{\kappa_1^2}{\kappa_2^2},
\]
 and thus $\f^*$ is the dual isothermic surface of $\f$.
 Then finally it is easy to show that 
\[
   \| \f - \f^*\| = \frac{1}{H^2}
\]
 holds, and therefore $\f$ (and thus $\f^*$) is a discrete isothermic 
 parametrized CMC surface by \eqref{eq:defCMC}.

\section{Proofs of Theorem \ref{thm:normalized} and
 Theorem \ref{thm:DPW}}\label{subsec:birk_quad}
 We now give the proof of {\rm Theorem \ref{thm:normalized}}.

 First apply the Birkhoff decomposition Theorem 
 \ref{Thm:Birkhoff} to 
 $\Phi$, i.e., 
\[
 \Phi = \Phi_+ \Phi_-, 
\]
 with $\Phi_{+} \in \LSLP$ and $\Phi_{-} \in \LSLNN$.
 We now define functions $\Delta_{\pm}$ as
\begin{equation}\label{eq:Deltapm}
 \Delta_{-} = \sqrt{1- H \alpha \l^{-2}}, \quad \mbox{and}\quad
 \Delta_{+} = \sqrt{1- H \bar \alpha \l^{2}}.
\end{equation}
 Let us compute the discrete Maurer-Cartan equation  of $\Phi_-$ multiplying $\Delta_-$.
 By using \eqref{eq:BirkhoffofF}, we have
\begin{equation}\label{eq:MCFm}
\Delta_- \Phi_{-,j} \Phi_{-, i}^{-1} = \Delta_- 
 \Phi_{+, j}^{-1} U_{ij}\Phi_{+, i},
\end{equation}
 where $\Phi_{\pm, i} = \Phi_{\pm}(v_i, \l)$.
 Noting that the expansion of $\Phi_{+, i}$ as
\[
 \Phi_{+, i} = \di (v^0_i, (v^0_i)^{-1}) + 
 * \l + \cdots,
\]
 we compute 
\[
\Delta_-  
 \Phi_{+, j}^{-1} U_{ij}\Phi_{+, i}
= 
\begin{pmatrix}
 \tilde d & \l^{-1} H \tilde u \\
 \l^{-1} \alpha \tilde u^{-1}   & \hat d
\end{pmatrix} + 
 * \l + \cdots , 
\]
 where $\tilde u$ takes values in $\C^{*}$. Since the left-hand 
side of \eqref{eq:MCFm} takes values in $\Lambda^-_* \SL$, we conclude 
\begin{equation}\label{eq:preL-}
 \Phi_{-,j} \Phi_{-, i}^{-1}
= \frac{1}{\Delta_-} \begin{pmatrix}
 1 &   \l^{-1} H \tilde u  \\
\l^{-1} \alpha \tilde u^{-1} & 1
\end{pmatrix}.
\end{equation}
 We now set $L_{-}(\mathfrak  e, \l)$ as the right-hand side of 
 the above equation and consider the compatibility condition 
 for $L_-$ on the
 elementary quadrilateral $(v_0, v_1, v_2, v_3)$, i.e.,
\[
 L_-(v_{1}, v_{2}, \alpha_{2}, \l) L_-(v_{0}, v_{1}, \alpha_{1}, \l)
 =  L_-(v_{3}, v_{2}, \alpha_{1}, \l) L_-(v_0, z_{3}, \alpha_2, \l).
\]
 Then a straightforward computation shows that this is equivalent with
\begin{gather*}
 \frac{\alpha_1}{\alpha_2} = \frac{\tilde u_{01} \tilde u_{32}(\tilde u_{12}-\tilde u_{03})}{\tilde u_{03} \tilde u_{12}(\tilde u_{32}-\tilde u_{01})}, \quad 
H\left(\frac{\tilde u_{01} \tilde u_{32}}{\tilde u_{03} \tilde u_{12}}
 -  \frac{\alpha_1}{\alpha_2}\right)  = 0\quad\mbox{and}\quad
H\left(\tilde u_{12} +\tilde u_{01} -\tilde u_{32} -\tilde 
 u_{03}\right) =0.
\end{gather*}
 By setting $\tilde u_{ij}$ by 
\[
 \tilde u_{ij} = z_i -z_j,
\]
 the above equations can be simplified as 
\[
 \frac{(z_0 -z_1)(z_3-z_2)}{(z_1 -z_2)(z_0 - z_3)}
 =\frac{\alpha_1}{\alpha_2},
\]
 which is the cross-ratio system \eqref{eq:def-cross}.
 Moreover, then $\Phi_j \Phi_i^{-1}$ in \eqref{eq:preL-} 
 is given as $L_-$ in \eqref{eq:Lminus}.
 This completes the proof. $\hfill \Box$

 Conversely, applying the Iwasawa decomposition to 
 $\Phi_-$, Theorem \ref{Thm:Iwasawa}, we have the extended frame.

 We finally give the proof of {\rm Theorem \ref{thm:DPW}}.
 Recall that $\Delta_{\pm}$ are functions defined in \eqref{eq:Deltapm}.
 Let us compute the discrete Maurer-Cartan equation for $\Phi$
 multiplying $\Delta_+ \Delta_-$, i.e.,
 $\Delta_+ \Delta_- \Phi_j\Phi_i^{-1}$: 
 By using \eqref{eq:IwasawaforFm}, we have
\[
 \Phi_j \Phi_i^{-1} =  
 \Phi_{+, j}^{-1} \Phi_{-, j} \Phi_{-, i}^{-1} 
 \Phi_{+,i}
 = \Phi_{+, j}^{-1} L_{-, ij} \Phi_{+, i},
\]
 and thus
\[
 \Delta_+ \Delta_-  \Phi_j \Phi_i^{-1} = \Delta_+ 
\Phi_{+, j}^{-1} \begin{pmatrix} 
 1 &  H (z_i-z_{j})\l^{-1}  \\
\dfrac{\alpha} {z_i-z_j} \l^{-1}  &  1
\end{pmatrix}\Phi_{+, i}.
\]
Then a straightforward computation of the right-hand side 
by using the expansions of $\Phi_{+, i} = \di (v^0_i, \;(v^0_i)^{-1})
 + *\l + \cdots$ and $\Delta_+ = 1 + *\l + \cdots $ shows that 
\[
 \Delta_+ \Delta_-  \Phi_j \Phi_i^{-1} = 
\begin{pmatrix} 
 (v^0_j)^{-1} v^0_i &  \l^{-1} H u_{ij}  \\
\l^{-1}\alpha u_{ij}^{-1}  &   v^0_j (v^0_i)^{-1}
\end{pmatrix}
 + * \l^{0} + \cdots,
\]
 where we set $u_{ij} = (z_i-z_j)v^0_j v^0_i$.
 Since $\Phi$ takes values $\LSU$ and by the form of $\Delta_{\pm}$, we have 
\[
 (\Delta_+ \Delta_-  \Phi_j \Phi_i^{-1})^* = (\Delta_+ \Delta_-)^{-1}  
 \Phi_j \Phi_i^{-1},
\] 
where $g(\l)^* = \left\{\overline{g (1/\bar \l)}^{T}\right\}^{-1}$,
 and therefore we conclude that 
\[
 \Delta_+ \Delta_-  \Phi_j \Phi_i^{-1} 
= 
\begin{pmatrix}
 d &  \l^{-1} H u - \l \alpha \bar u^{-1} \\
\l^{-1} \alpha u^{-1}- \l H \bar  u   & \bar d
\end{pmatrix}.
\]
 Thus $\Phi_j \Phi_i^{-1}$ is $U(\mathfrak e, \l)$
 as in \eqref{eq:dextendedU}. This completes the proof.
$\hfill \Box$

\subsection*{Acknowledgments}
We thank Wai Yeung Lam for bringing this topic to our attention and insightful discussion. 
\printbibliography

\end{document}